\documentclass[12pt]{amsart}

\usepackage{cancel}
\usepackage{latexsym, amssymb, amsmath,mathtools}
\usepackage{soul,esint}
\usepackage{amsfonts, graphicx}
\usepackage{graphicx,color}
\newcommand{\be}{\begin{equation}}
\newcommand{\ee}{\end{equation}}
\newcommand{\beq}{\begin{eqnarray}}
\newcommand{\eeq}{\end{eqnarray}}

\newtheorem{thm}{Theorem}[section]

\newtheorem{lma}{Lemma}[section]
\newtheorem{prop}{Proposition}[section]
\newtheorem{cor}{Corollary}[section]
\newtheorem{defn}{Definition}[section]

\newtheorem{claim}{Claim}[section]
\theoremstyle{remark}
\newtheorem{rem}{Remark}[section]
\numberwithin{equation}{section}

\def\be{\begin{equation}}
\def\ee{\end{equation}}
\def\bee{\begin{equation*}}
\def\eee{\end{equation*}}

\def\lf{\left}
\def\ri{\right}

\def\K{K\"ahler }
\def\KR{K\"ahler-Ricci }

\def\Ric{\text{\rm Ric}}
\def\Rm{\text{\rm Rm}}
\def\OB{\text{\rm OB}}

\def\wt{\widetilde}

\def\p{\partial}
\def\ddb{\sqrt{-1}\partial\bar\partial}
\def\aint{\frac{\ \ }{\ \ }{\hskip -0.4cm}\int}

\def\heat{\lf(\frac{\p}{\p t}-\Delta\ri)}

\def\e{\varepsilon}
\def\a{{\alpha}}
\def\b{{\beta}}

\def\ijb{{i\bar{j}}}

\def\C{\mathbb{C}}

\def\tR{ \mathcal{R}}

\begin{document}
\title[]
{K\"ahler manifolds with almost non-negative curvature}

 \author{Man-Chun Lee}
\address[Man-Chun Lee]{Department of Mathematics, Northwestern University, 2033 Sheridan Road, Evanston, IL 60208}
\email{mclee@math.northwestern.edu}

\author{Luen-Fai Tam}
\address[Luen-Fai Tam]{The Institute of Mathematical Sciences and Department of
 Mathematics, The Chinese University of Hong Kong, Shatin, Hong Kong, China.}
 \email{lftam@math.cuhk.edu.hk}

\renewcommand{\subjclassname}{
  \textup{2010} Mathematics Subject Classification}

\date{October 2019; revised in \today}

\begin{abstract}
In this paper, we construct local and global solutions to the \KR flow from a non-collapsed \K manifold with curvature bounded from below. Combines with the mollification technique of McLeod-Simon-Topping, we show that the Gromov-Hausdorff limit of sequence of complete noncompact non-collapsed \K manifolds with orthogonal bisectional curvature and Ricci curvature bounded from below is homeomorphic to a complex manifold. We also use it to study the complex structure of complete \K manifolds with nonnegative orthogonal bisectional curvature, nonnegative Ricci curvature and maximal volume growth.
\end{abstract}


\maketitle

\markboth{Man-Chun Lee and Luen-Fai Tam}{\K manifolds with almost non-negative curvature}

\section{introduction}

In this work, we are interested to study complete non-compact \K manifolds with curvature bounded from below. More specifically, we consider complete non-compact \K manifolds with both {\it orthogonal bisectional curvature} and {\it Ricci curvature} being bounded from below. Let us first recall the definition of orthogonal bisectional curvature:
\begin{defn}\label{d-OB}
On a \K manifold $(M,g)$, we say that the orthogonal bisectional curvature $\OB$ is bounded from below by $k$ (denoted   by $\OB\geq k$) if for any $X,Y\in T^{1,0}M$, $g(X,\bar Y)=0$, we have
\begin{align}
\label{OB-defn}R(X,\bar X,Y,\bar Y) \geq k B(X,\bar X,Y,\bar Y)
\end{align}
where $B(X,\bar Y, Z,\bar W)=g(X,\bar Y)g(Z,\bar W)+g(X,\bar W)g(Z,\bar Y).$
\end{defn}

 In Riemannian geometry, the Gromov's Compactness theorem \cite{Gromov1999} states as follows: {\it Let $(M_i^m,g_i,p_i)$ be a sequence of pointed complete Riemannian manifolds of the same  dimension $m$ with Ricci curvature bounded from below by $k$. Then after passing to a subsequence $(M_i, p_i)$,  together with    the distance function  $d_i$ induced by $g_i$  will converge to a complete pointed metric space $(M_\infty,d_\infty,p_\infty)$ in the pointed Gromov-Hausdorff (PGH) sense}.

For the definition of pointed Gromov-Hausdorff convergence, we refer readers to \cite[Definition 8.1.1]{BYS}.

 One of the main question is to study the structure of the pointed Gromov-Hausdorff limit $(M_\infty,d_\infty,p_\infty)$.
   In this work, we would like to study the limit space in the case when each $(M_i,g_i,p_i)$ is \K with $\OB(g_i)$ and  $ \Ric(g_i)$ both being bounded from below by a constant $k$ independent of $i$ and is non-collapsing in the sense that $V_{g_i}(p_i,1)=Vol_{g_i}\left(B_{g_i}(p_i,1) \right)\ge v$  for some $v>0$ for all $i$. By rescaling, we may assume $k=-1$ and we obtain the following:

\begin{thm}\label{t-intro-GH}\hfill

\begin{enumerate}
\item[{\bf(I)}]
Suppose $(M^n_i,g_i)$ is a sequence of pointed complete non-compact \K manifolds with $p_i\in M_i$,  such that for all $i\in \mathbb{N}$
\begin{enumerate}
\item $\mathrm{OB}(g_i)\geq -1$ on $M_i$;
\item $\mathrm{Ric}(g_i)\geq -1$ on $M_i$;
\item $V_{g_i}(p_i,1)\geq v>0$ for some $v>0$ independent of $i$ ({\bf weakly non-collapsed}).
\end{enumerate}
Then there exist a   complex  manifold $(M_{\infty},J_\infty)$, $p_{\infty}\in M_\infty$ and a complete distance metric $d_\infty : M_{\infty}\times M_{\infty} \rightarrow [0,+\infty)$ generating the same topology as $M_\infty$ such that after passing to a subsequence in $i$ we have
$$(M_i,d_{g_i},p_i)\rightarrow (M_\infty,d_\infty ,p_\infty)$$
in the pointed Gromov-Hausdorff sense.   Moreover, there exist an increasing sequence   $a_l>0$, a decreasing sequence   $T_l>0$ and a smooth \KR flow solution $g_\infty(t)$ defined on $\bigcup_{l=1}^\infty B_{d_\infty}(p_\infty,l)\times (0,T_l]$ such that for all $l\in\mathbb{N}$,
\begin{equation*}
\left\{
\begin{array}{ll}
&\Ric(g_\infty(t))\geq -a_{l};\\
&\OB(g_\infty(t))\ge -a_{l};\\
&|\Rm(g_\infty(t))|\leq a_{l}t^{-1};\\
&\mathrm{inj}_{g_\infty(t)}(x)\geq \sqrt{a_{l}^{-1}t}
\end{array}
\right.
\end{equation*}
on $B_{d_\infty}(p_\infty,l)\times (0,T_l]$.

\item[{\bf(II)}] Suppose in addition,
$V_{g_i}(x,1)\geq v>0$ for some $v>0$ independent of $i$ and $x\in M_i$ ({\bf uniformly non-collapsed}), then we can choose $a_l$ and $T_l$ to be constants independent of $l$. In particular, the limit space $M_\infty$ admits a complete \K metric $g_\infty$ with bounded curvature so that the distance metric induced by $g_\infty$ is quasi-isometric to $d_\infty$.
\end{enumerate}
\end{thm}

We will use techniques in \KR flow to prove these results following the ideas by Simon-Topping \cite{SimonTopping2017}, Bamler--Cabezas-Rivas--Wilking \cite{BCRW}, Hochard \cite{Hochard2019}, Lai \cite{Lai2019} and McLeod-Topping \cite{McLeodTopping2018,McLeodTopping2019}, where limit spaces of Riemannian manifolds have been studied using Ricci flow. Their results say that the limit spaces are smooth manifolds under conditions on lower bounds of various curvatures together with non-collapsing. By  the stability theorem of Perelman \cite{Perelman1991}, the same result is true under the condition that the sectional curvature is bounded from below. In the \K setting,
Donaldson-Sun \cite{DonaldsonSun2014} proved
that the Gromov-Hausdorff limit of a sequence of non-collapsed, polarized compact K\"ahler
manifolds with bounded Ricci curvature,  is a normal projective variety. It was generalized by Liu-Sz\`ekelyhidi in \cite{LiuSzekelyhidi2018} by   removing the Ricci upper bound.     For non-compact \K manifolds, using  Gromov-Hausdorff convergence theory by Cheeger-Colding,  assuming that bisectional curvature of the sequence satisfies $\mathrm{BK}(g_i)\geq -1$ together with non-collapsing, Liu   \cite{Liu2017,Liu2018} proved that the limit space  is homeomorphic to a normal complex analytic space so that the singular set is  of complex codimension at least $4$.  Recall that a \K metric is said to have bisectional curvature bounded from below by $k$ (denote it by $\mathrm{BK}\geq k$) if \eqref{OB-defn} holds for any $X,Y\in T^{1,0}M$. Clearly, $\mathrm{BK} \geq k$ implies $\OB\geq k$ and $\Ric\geq (n+1)k$. Part {\bf(I)} of the theorem is a generalization  of Liu's result.  { This can also be viewed as complex analogue to the results in  Riemannian case mentioned above.}

From the arguments in \cite{McLeodTopping2018,McLeodTopping2019}, in order to prove Theorem \ref{t-intro-GH}         {\bf (I)}, we need to construct a pyramid solution to the  \KR flow which is in parallel with \cite[Theorem 1.3]{McLeodTopping2019}. In order to prove Theorem \ref{t-intro-GH}{\bf(II)}, we need to construct a global solution   with precise estimates on lifespan, curvature and injectivity radius. In this regard, we have the following:

\begin{thm}\label{t-intro-pyramid} \

\begin{enumerate}
  \item [{\bf(I)}]
  For any $n, v_0>0$, there exist  non-decreasing sequences   $a_k,\b_k\geq 1$ and non-increasing sequence $S_k>0$ such that the following holds: Suppose $(M^n,g_0)$ is a complete non-compact \K manifold and $p\in M$ so that
\begin{enumerate}
\item $\Ric(g_0)\geq -1$ on $M$;
\item $\OB(g_0)\geq -1$ on $M$;
\item $V_{g_0}(p,1)\geq v_0$.
\end{enumerate}
Then for any $m\in \mathbb{N}$, there is a solution to the \KR flow $g_m(t)$ defined on a subset $D_m$ of space-time given by
$$D_m=\bigcup_{k=1}^m \left(B_{g_0}(p,k)\times [0,S_k]\right),$$
with $g_m(0)=g_0$ such that $g_m(t)$ satisfies
\bee
\left\{
\begin{array}{ll}
&\Ric(g_m(t))\geq -\b_k;\\
&\OB(g_m(t))\ge -\b_k;\\
&|\Rm(g_m(t))|\leq a_kt^{-1};\\
&\mathrm{inj}_{g_m(t)}(x)\geq \sqrt{a_k^{-1}t}
\end{array}
\right.
\eee
on each $B_{g_0}(p,k)\times (0,S_k]$.

\item [{\bf(II)}] If in addition, $V_{g_0}(x,1)\geq v_0>0$ for all $x\in M$, then
  there is $a(n,v_0)$, $T(n,v_0)$, $L(n,v_0)>0$ and a complete solution to the \KR flow $g(t)$ defined on $M\times [0,T]$ so that for all $(x,t)\in M\times (0,T]$,
  \bee
\left\{
\begin{array}{ll}
&\mathrm{Ric}(g(t))\geq -L;\\
&\mathrm{OB}(g(t))\geq -L;\\
&|\Rm(g(t))|\leq at^{-1};\\
&\mathrm{inj}_{g(t)}(x)\geq \sqrt{a^{-1}t}.
\end{array}
\right.
\eee

\end{enumerate}
\end{thm}

  The pyramid \KR flow in part {\textbf{(I)} of the above theorem has uniform estimates on geodesic balls of any fixed radius. This enables us to take local Hamilton's compactness to construct local smooth limits in the proof of Theorem \ref{t-intro-GH}. The global limiting manifold $M_\infty$ in Theorem \ref{t-intro-GH} will be obtained by gluing all local smooth limits.

To construct the pyramid solutions in the theorem, the main ingredient is to construct local solutions to the \KR flow on any geodesic balls with uniform estimates on lifespan, curvature and injectivity radius. This can be done by using techniques of Chern-Ricci flows on Hermitian manifolds which were introduced by Gill \cite{Gill2011} and Tosatti-Weinkove \cite{TosattiWeinkove2015} and extending   the local construction  and local estimates in \cite{LeeTam2017-1},  by adopting the Hochard's idea \cite{Hochard2016} of partial Ricci flow to the Chern-Ricci flow setting.   This improves the result in \cite{LeeTam2017}. One of the main new ingredient is a local maximum principle by Hochard \cite{Hochard2019} building on the heat kernel estimates in \cite{BCRW}.     Modifying the argument by Simon-Topping \cite{SimonTopping2017} as in \cite{LeeTam2017}, we obtain the following extension result which will be used to construct pyramid \KR flows:

\begin{thm}\label{t-intro-local-ext-KRF}
For all $\b_0 ,B\geq 1$ and $v>0$, there exist $a(n,v,\b_0)\geq 1$ and $T_0(n,v,\b_0,B)>0$   such that  the following is true: Suppose $(M^n,g_0)$ is a \K manifold with complex dimension $n$. Let $p\in M$ so that $B_{g_0}(p,R+2)\Subset M$ for some $R\geq 1$ and for all $x\in B_{g_0}(p,R+1)$,
\begin{enumerate}
\item [(i)]$\Ric(g_0)\geq -\b_0$;
\item [(ii)]$\OB(g_0)\geq -\b_0$;
\item [(iii)] $V_{g_0}(x,1)\geq v$.
\end{enumerate}
If there is $S>0$ and a smooth \KR flow solution $\tilde g(t)$ defined on $B_{g_0}(p,R+1)\times [0,S]$ with $\tilde g(0)=g_0$ and satisfies $|\Rm(\tilde g(t))|\leq Bt^{-1}$ and $\mathrm{inj}_{\tilde g(t)}(x)\geq \sqrt{B^{-1}t}$, then there is a smooth \KR flow $g(t)$ solution on $B_{g_0}(p,R)\times [0,T_0]$ such that $g(t)=\tilde g(t)$ on $B_{g_0}(p,R)\times [0,S\wedge T_0]$ with
\begin{enumerate}
\item $|\Rm(g(t))|\leq at^{-1}$;
\item $ \mathrm{inj}_{g(t)}(x)\geq \sqrt{a^{-1}t}$
\end{enumerate}
for all $(x,t)\in B_{g_0}(p,R)\times [0,T_0]$.
\end{thm}
In particular, the theorem implies a local existence of \KR flow on geodesic balls, see Lemma \ref{l-local} for details. Theorem \ref{t-intro-pyramid} {\bf(II)} also has an  application on the complex structure of complete non-compact \K manifold with $\OB\ge 0$, $\Ric\ge0$ and of maximal volume growth.

\begin{cor}\label{uniformization}
Suppose $(M^n,g_0)$ is a complete non-compact \K manifold with non-negative orthogonal bisectional curvature, non-negative Ricci curvature and maximal volume growth, then $M$ is biholomorphic to a pseudoconvex domain in $\mathbb{C}^n$ which is homeomorphic to $\mathbb{R}^{2n}$. Moreover, $M$ admits a non-constant holomorphic function with polynomial growth.
\end{cor}

The first part of corollary is related to the uniformization conjecture by Yau which states that a complete non-compact \K manifold with positive bisectional curvature must be biholomorphic to $\mathbb{C}^n$. In general, $\mathrm{BK}\ge 0$ implies $\OB\ge 0, \Ric\ge 0$. However, the converse is false in general. In particular, Ni and Zheng had constructed examples of complete $U(n)$-invariant \K metric which has non-negative Ricci curvature, non-negative orthogonal bisectional curvature and maximal volume growth but has negative holomorphic sectional curvature somewhere, see \cite[section 7]{NiZheng2018}.  Under $\mathrm{BK}\geq 0$, the uniformization conjecture of Yau was proved recently by Liu \cite{Liu2017} under an extra assumption of maximal volume growth, see also \cite{ChauTam2006,LeeTam2017} for a different approach using the \KR flow. It will be interesting to know if one can obtain the full biholomorphism to $\mathbb{C}^n$ under this weaker assumption. Note that the condition on Corollary \ref{uniformization}
is strictly weaker than that in \cite{Liu2017,LeeTam2017}. We also would like to remark that in a recent work of Liu-Sz\'ekelyhidi \cite{LiuSzekelyhidi2018}, it was shown that a complete non-compact \K manifold is biholomorphic to $\mathbb{C}^n$ under $\Ric\geq 0$ and almost Euclidean asymptotic volume ratio.  The second part of the corollary is a generalization of \cite[Theorem 1.4]{Liu2016} by Liu which states that a complete non-compact \K manifold with non-negative bisectional curvature and maximal volume growth supports a non-trivial holomorphic function with polynomial growth.

The paper is organized as follows:  In section 2, we will derive some a-priori estimates for the \KR flow. In section 3, we will construct local and global solutions to the \KR flow { and prove Theorem \ref{t-intro-local-ext-KRF} and Theorem \ref{t-intro-pyramid}}. In section 4, we will use the solutions of the \KR flow to prove Theorem \ref{t-intro-GH} and Corollary \ref{uniformization}. In the appendix, we will collect some useful  results which will be used in the main part of the paper. We will also show that the examples constructed by Ni-Zheng has $\OB,\Ric\geq 0$ and maximal volume growth but $\mathrm{BK}<0$ somewhere.

{\it Acknowledgement}: The authors are grateful to {Rapha\"el} Hochard for sending us his thesis and generously sharing his ideas.   The authors would like to thank the referee for some useful comments.   Part of the works was done when the first author visited the Institute of Mathematical Sciences at The Chinese University of Hong Kong, which he would like to thank for the hospitality. M.-C. Lee is supported in part by NSF grant 1709894. L.-F. Tam is supported in part by Hong Kong RGC General Research Fund \#CUHK 14301517.

}

\section{curvature estimates}
In this section, we will derive   local curvature estimates of the \KR flow.   We will prove that under  some assumptions which are invariant under parabolic rescaling, the  almost non-negativity of the orthogonal bisectional curvature and   Ricci curvature will be preserved along the flow locally.  We will first need the following lemma from \cite[Theorem 1.1]{Niu2014} stating that almost non-negativity of orthogonal bisectional curvature will imply almost non-negativity of scalar curvature.
\begin{lma}\label{OB-R-lma}
Let $(M,g)$ be a \K manifold and $p\in M$. If $\OB(g(p))\geq -\mu$ for some $\mu \geq 0$. Then the scalar curvature $R_g(p)\geq -C_n\mu$ for some dimensional constant $C_n>0$.
\end{lma}

Now we will show that the lowest eigenvalue of Ricci curvature can be controlled along the flow if the orthogonal bisectional curvature is bounded from below by some constant $-k$ along the flow. In the compact case with $k=0$, it was proved by Chen \cite{Chen2007}.\vskip .2cm
 
 For notational convenience, we will use $a\wedge b$ to denote $\min\{a,b\}$ for $a,b\in \mathbb{R}$.

\begin{prop}\label{Ric-pre}
Suppose $(M^n,g(t)),t\in [0,T]$ is a smooth solution to the \KR flow such that for some $p\in M$, $r>0$, we have $B_t(p,r)\Subset M$ for all $t\in [0,T]$. If there is $a,\mu_0,k>0$ such that
\begin{enumerate}
\item $|\mathrm{Rm}(x,t)|\leq at^{-1}$ on $B_t(p,r)$, $t\in (0,T]$;
\item $\Ric(g(0))\geq -\mu_0r^{-2}$ on $B_{g_0}(p,r)$;
\item $\mathrm{OB}(g(t))\geq -kr^{-2}$ on $B_t(p,r)$, $t\in [0,T]$.
\end{enumerate}
Then there is $ T_1(n,a,\mu_0,k)>0$ depending only on $n$ and the upper bounds of $a, k, \mu_0$ such that for all $x\in B_t(p,\frac{r}{8})$ and $t\in [0,T\wedge  (r^2T_1)]$,
$$\Ric(x,t)\geq -c_nr^{-2}\lf(k( a+kt)+\mu_0+1\ri)g_\ijb .$$
{ Here we denote $a\wedge b=\min\{a,b\}$.}
\end{prop}
\begin{proof}
Here and below, $c_i,i=1,2,...$ will denote distinct positive constants depending only on $n$. By parabolic rescaling, we may assume $r=1$.
By \cite[Lemma 8.3]{Perelman2002} and the curvature assumption, the distance function $d_t(x,p)$ satisfies
\begin{align}
\heat d_t(x,p)\geq -\frac{c_2 a}{\sqrt{t}}
\end{align}
in the sense of barrier whenever $d_t(x,p)\geq \sqrt{t}$. Let $\phi$ be a cutoff function on $[0,+\infty)$ such that $\phi$ is identical $1$ on $[0,\frac{1}{4}]$, vanishes outside $[0,\frac{3}{4}]$ and satisfies
\begin{align}
100\leq \phi'\leq 0,\;\; \phi'' \geq -100.
\end{align}
Let $\Phi(x,t)=e^{-100mt}\phi^m(\eta(x,t))$ where $\eta(x,t)=d_t(x,p)+2c_2 a \sqrt{t}$ and $m\in \mathbb{N}$ is a large integer to be fixed later. Then the cutoff function $\Phi$ satisfies
\begin{align}\label{cut}
\heat \Phi \leq 0
\end{align}
in the sense of barrier. We may assume $\Phi$ to be smooth when we apply maximum principle, see \cite[section 7]{SimonTopping2016} for detailed exposition (see also \cite{HuangTam2018}).

By Lemma \ref{OB-R-lma}, we have $R_{g(t)}\geq -c_1k  $ in $B_t(p,1)$ for $t\in [0,T]$. Let
\be\label{e-tR}
\wt R=R+c_1k\ge0
\ee

For any $\e>0$, let $\mu(x,t)=kt \wt R  +\frac{1}{4}t^\frac{1}{4}+\mu_0+\e$ and consider the modified Ricci tensor $A_{i\bar j}=\Phi R_{i\bar j}+\mu g_{i\bar j}$.
We want to prove that $A_\ijb\ge0$ on $[0,T_1\wedge T]$, for some $T_1(n,k,a,\mu_0)>0$ depending only on $n$ and the upper bounds of $k, a, \mu_0$.

Clearly, $A_{i\bar j}>0$ for $t$ sufficiently small and outside the support of $\Phi$.
Hence if $A_\ijb\le 0$ somewhere, then there is   $t_0\in (0,T\wedge  T_1]$ and $x_0$ in which $\Phi(x_0,t_0)>0$,  $u\in T^{1,0}_{x_0}M$ so that  be such that $A>0$ on $[0,t_0)$ and for so that $A_{u\bar u}(x_0,t_0)=0$. We may assume $g_{u\bar u}(x_0,t_0)=1$ by rescaling. Extend $u$ around $(x_0,t_0)$ such that $\nabla u=0$ and $\heat u =0$ at $(x_0,t_0)$. Then at $(x_0,t_0)$, we have
\begin{equation}\label{e-A-1}
\begin{split}
0\geq &\heat A_{u\bar u}\\
=& R_{u\bar u}\heat \Phi   + \Phi \heat R_{u\bar u}-2{\bf Re}\left(g^{i\bar j} \Phi_i \cdot \partial_{\bar j}R_{u\bar u} \right)\\
&+\heat \mu -\mu R_{u\bar u}.\\
\end{split}
\end{equation}

We want to estimate each terms in the last line above.
Since $A_{u\bar u}=0$ at $(x_0,t_0)$, $R_{u\bar u}$ must be negative. Combines this with \eqref{cut}, we have
\be\label{e-A-2}
 R_{u\bar u}\heat \Phi\ge0.
\ee

At $(x_0,t_0)$, we choose an unitary frame $e_i$ such that $e_1=u$ and $R_{i\bar j}=\lambda_i \delta_{ij}$ with $\lambda_n\geq \lambda_{n-1}\geq \cdots \geq \lambda_1=R_{1\bar 1}$. This is possible because $u$ is the eigenvector of $A$ corresponding to the lowest eigenvalue $0$ at $(x_0,t_0)$. Hence,
\begin{equation}\label{e-A-3}
\begin{split}
\Phi \heat R_{1\bar 1}-\mu R_{1\bar 1}
=&\Phi \sum_{i=1}^n R_{1\bar 1 i\bar i} \lambda_i -\Phi\sum_{j=1}^n R_{1\bar j}R_{j\bar 1}-\mu\lambda_1\\
=&\Phi \sum_{i=1}^n (R_{1\bar 1i\bar i}+k)\lambda_i-\Phi k \sum_{i=1}^n \lambda_i-\Phi \lambda_1^2-\mu \lambda_1\\
\geq &\lambda_1 \Phi \sum_{i=1}^n (R_{1\bar 1i\bar i}+k)-\Phi k R\\
=& \Phi \lambda_1^2+ \Phi k \lambda_{1}-\Phi k R \\
\geq &-\Phi k R -\frac{1}{4} k^2.
\end{split}
\end{equation}

Here we have used $A_{1\bar 1}=0$ and hence $\Phi \lambda_1+\mu=0$ at $(x_0,t_0)$. Using  the fact that $\nabla A_{u\bar u}=0$ and $A_{u\bar u}=0$ at $(x_0,t_0)$, we have
\begin{equation}\label{e-A-4}
\begin{split}
-2{\bf Re}\left(g^{i\bar j} \Phi_i \cdot \partial_{\bar j}R_{u\bar u} \right)=&2\frac{|\nabla \Phi|^2}{\Phi }R_{u\bar u}+\frac{2}{\Phi}{\bf Re}\left(g^{i\bar j}\mu_i \Phi_{\bar j} \right)\\
\geq &-2\frac{|\nabla \Phi|^2}{\Phi^2} \mu -2\frac{|\nabla \Phi||\nabla \mu|}{\Phi}\\
\geq & -\frac{c_3 m^2 \mu}{\Phi^\frac{2}{m}}-\frac{c_3m|\nabla \mu|}{\Phi^\frac{1}{m}}\\
\ge&-\frac{c_4 a^\frac2m m^2 \mu_0^{1-\frac2m}}{t_0^\frac{2}{m}}-
\frac{c_4ma^\frac1m|\nabla \mu|}{t_0^{\frac1m(1+\frac14)}}
\end{split}
\end{equation}
where we have used the fact that    at $(x_0,t_0)$,
$
-\Phi R_{1\bar 1}=\mu$
and the assumption $|\Rm(g(t))|\le a/t$  at $(x_0,t_0)$ so that $$\Phi\ge ca^{-1}t_0\mu\ge \frac14 ca^{-1}t_0^{1+\frac14}$$ at $(x_0,t_0)$, for some constant $c$ depending only on $n$.

\be\label{e-A-5}
\begin{split}
\heat \mu=&k\wt R+t_0^{-\frac34}+kt_0|\Ric|^2
\ge k\wt R+t_0^{-\frac34}.
\end{split}
\ee
On the other hand,   by Shi's estimate \cite{Shi1989} (see also \cite[Theorem 1.4]{CaoChenZhu2008}), we may assume that $|\nabla \mathrm{Rm}|\leq C(n,a)t^{-3/2}$ on $B_t(p,\frac{3}{4})$, $t\in (0,T\wedge  T_1]$.
Hence by \eqref{e-A-1}--\eqref{e-A-5}, at $(x_0,t_0)$ we have
\be\label{e-concl-sum I}
\begin{split}
0\ge & -\Phi k R -\frac{1}{4} k^2-\frac{c_4 a^\frac2m m^2 \mu_0^{1-\frac2m}}{t_0^\frac{2}{m}}-
\frac{c_4ma^\frac1m|\nabla \mu|}{t_0^{\frac1m(1+\frac14)}}+k\wt R+t_0^{-\frac34}\\
\ge&-\frac14k^2-\frac{c_4 a^\frac2m m^2 \mu_0^{1-\frac2m}}{t_0^\frac{2}{m}}-
\frac{D_1 c_4ma^\frac1m  }{t_0^{\frac1m(1+\frac14)+\frac12}} +t_0^{-\frac34}\\
\end{split}
\ee
where $D_1$ is a constant depending only on $n,a$.

Therefore, if we choose $m=10$, then \eqref{e-concl-sum I} implies $t_0\geq T_1(n,k,\mu_0,a)$ which is a positive constant depending only on the $n$ and the upper bounds of $k, \mu_0, a$. Hence we have
$
A_\ijb\ge0
$
on $B_t(p,\frac18)$ and $t\in [0,T_1\wedge T]$. Since $|\Rm(g(t))|\le at^{-1}$, we conclude that in $B_t(p,\frac18)$ and $t\in [0,T_1\wedge T]$, we have
\bee
\begin{split}
R_\ijb\ge &-\mu g_\ijb\\
\ge &\lf(-tk(c_5a\frac1t+c_1k)-\mu_0-\frac14 t^\frac14-\e\ri)g_\ijb\\
=&-c_6\lf(k( a+kt)+\mu_0+1\ri)g_\ijb-\e g_\ijb.
\end{split}
\eee
provided $T_1\le 1$.
This completes the proof by letting $\e\rightarrow 0$.
\end{proof}

{ Before we state the next curvature estimate, let us recall the following local maximum principle, which was implicitly proved in \cite[Proposition II.2.6]{Hochard2019}. For a more direct proof, we refer interested readers to \cite[Section 4]{LeeTam2020}.

\begin{lma}\label{OB-pre-Hochard}
For $m\in \mathbb{N}, \a>0, c_0>0$, there is $A(m,\a, \lambda)>0$ such that the following holds: Suppose $M^m$ is a $m$-dimensional manifold and $g(t)$ is a smooth solution to the Ricci flow with $g(0)=g_0$. Let $p\in M$ so that $B_{g_0}(p,1)\Subset M$. Suppose for all $(x,t)\in B_{g_0}(p,1)\times [0,T]$,
\begin{enumerate}
\item $\Ric(g(0))\geq -1$;
\item $|\Rm(g(t))|\leq \a t^{-1}$;
\item $\mathrm{inj}_{g(t)}(x)\geq \sqrt{\a^{-1}t}$.
\end{enumerate}
Then for any nonnegative Lipschitz function $\varphi$    on $B_{g_0}(p,1)\times [0,    T]$,  satisfying
$$
\heat \varphi\le \tR\varphi+\lambda\varphi^2,
$$
 in the sense of barrier such that
 \bee
 \left\{
   \begin{array}{ll}
        \varphi \leq \a t^{-1}, &\hbox{on $B_{g_0}(p,1)\times(0,T]$}; \\
     \varphi(0)\le 1, &\hbox{on $B_{g_0}(p,1)$.}
   \end{array}
 \right.
 \eee
 where  $\tR(x,t)$ is the scalar curvature of $g(t)$. Then   $\varphi(x,t)\le A\rho^{-2}(x)$ for $(x,t)\in B_{g_0}(p,1)\times [0,T ]$
where
  $\rho(x)=\sup\{ r\in (0,1]: B_{g_0}(x,r)\Subset B_{g_0}(p,1)\}.$
\end{lma}

}

  We have the following local persistence of curvature conditions. In the complete bounded curvature case, the preservation of nonnegative orthogonal bisectional was first discovered by Cao-Hamilton, see \cite{Cao2013}.
  \begin{prop}\label{OB-esti}
For any $n,L>0$, there exist  $C_0(n,L)>1,  T_2(n,L)>0$ such that the following holds: Suppose $(M^n,g(t)),t\in [0,T]$ is a smooth solution to the \KR flow with complex dimension $n$. Let $p\in M$ be such that $\mathrm{OB}(g_0(x))\geq -kr^{-2}$ and $\Ric(g_0)(x)\geq -kr^{-2}$ on $B_{g_0}(p,k^{-\frac12}r)\Subset M$ for some $k>0$. Suppose that for all $(x,t)\in B_{g_0}(p,k^{-\frac12}r)\times (0,T]$, we have
\begin{enumerate}
\item  $|\mathrm{Rm}(x,t)|\leq Lt^{-1}$;
\item   $\mathrm{inj}_{g(t)} (x)\geq \sqrt{L^{-1}t}$.
\end{enumerate}
Then  for all $t\in T\wedge  (k^{-1}r^2 T_2)$, we have
\begin{enumerate}
\item[(a)]  $\mathrm{OB}\left (g(p,t)\right)\geq -C_0(n,L)kr^{-2}$;
\item[(b)]  $\Ric(g(p,t))\ge -C_0(n,L)kr^{-2}$.
\end{enumerate}
\end{prop}
\begin{proof}  By rescaling, we may assume that $kr^{-2}=1$. For $(x,t)\in B_{g_0}(p,1)\times [0,T]$, let
\bee
\ell(x,t)=\inf\{l\ge0|\ \mathrm{OB}(g_t(x,t))\ge -l\}.
\eee
By \cite[(6.3)]{NiLi2019}, $\ell(x,t)$ satisfies
\be\label{e-NiLi}
\heat \ell\le {\tR}\ell+c_n\ell^2
\ee
in the sense of barrier. Here $\tR(x,t)$ is the scalar curvature and $c_n$ is a positive constant depending only on $n$. Then $(a)$ follows from Lemma \ref{OB-pre-Hochard}. And $(b)$ follows from $(a)$, the shrinking ball Lemma \ref{l-balls} and Proposition \ref{Ric-pre}.
\end{proof}

By letting $k\to 0$, we have the following:

\begin{cor}\label{c-NOB} Let $(M^n,g(t)),t\in [0,T]$ be a smooth   complete  solution to the \KR flow with complex dimension $n$ with $g(0)=g_0$. Suppose $\OB(g_0)\ge0, \Ric(g_0)\ge 0$ and suppose
$$
|\Rm(g(t))|\le Lt^{-1},\   \mathrm{inj}_{g(t)}(x)\ge \sqrt{L^{-1}t}
$$
for all $(x,t)\in M\times(0,T]$. Then $\OB(g(t))\ge 0, \Ric(g(t))\ge 0$.
\end{cor}
\begin{rem} By modifying the method in \cite{LeeTam2017-1}, one can prove that the nonnegativity of orthogonal bisectional curvature will be preserved under the \KR flow $g(t)$ which satisfies $|\Rm(g(t))|\le a/t$ for $t>0$ without the assumption on injectivity radius.
\end{rem}

The following curvature estimates of \KR flow can be proved using work by Ni \cite{Ni2005} and Ni-Li \cite{NiLi2019} on the classification of ancient \KR flow solutions with non-negative curvature.
\begin{prop}\label{pseudo-OB}
For any $n,v>0$, there is $ T_3(n,v), L(n,v)>0$ so that  the following holds. Suppose $(M^n,g(t)),t\in [0,T]$ is a smooth solution to the \KR flow (not necessarily complete) such that for some $p\in M$, $r>0$, we have $B_t(p,r)\Subset M$ for all $t\in [0,T]$ and
\begin{enumerate}
\item[(a)] $V_{g_0}(p,r)\geq vr^{2n}$;
\item[(b)] $\Ric(g(t))\geq -r^{-2}$ on $B_t(p,r)$ for all $t\in [0,T]$;
\item[(c)] $\mathrm{OB}(g(t))\geq -r^{-2}$ on $B_t(p,r)$ for all $t\in [0,T]$.
\end{enumerate}
Then for all $t\in [0,T\wedge ( r^2T_3)]$ and $x\in B_{t}(p,\frac{r}{2})$,
$$|\mathrm{Rm}(x,t)|\leq \frac{L}{t}.$$
Moreover for $x\in B_t(p,\frac{r}{4})$ and $t\in [0,T\wedge (r^2T_3)]$,
\begin{enumerate}
\item  $|\nabla \mathrm{Rm}(x,t)|\leq {L}{t^{-3/2}}$
\item $\mathrm{inj}_{g(t)}(x)\geq \sqrt{L^{-1}t}$
\end{enumerate}
\end{prop}
\begin{proof}
The proof on the estimates of $|\Rm(g(t))|$ follows verbatim from that in \cite[Lemma 2.1]{SimonTopping2016}.  For reader's convenience, we sketch the proof here. We first establish the estimate of $|\Rm(x,t)|$. By parabolic rescaling, it suffices to consider the case $r=1$. Suppose the conclusion is false, then for any $L_k\rightarrow +\infty$, there exists sequence of \KR flow $(M^n_k,g_k(t),p_k)$, $t\in [0,T_k]$ so that the curvature conclusion fail in an arbitrarily short time. We may assume $L_kT_k\rightarrow 0$. By smoothness of each \KR flow, we can choose $t_k\in (0,T_k]$ so that
\begin{enumerate}
\item[(i)] $B_{g_k(t)}(p_k,1)\Subset M_k$ for $t\in [0,t_k]$;
\item[(ii)] $V_{g_k(0)}(p_k,1)\geq v$;
\item[(iii)] $\Ric(g_k(t))\geq -1$ on $B_{g_k(t)}(p_k,1)$, $t\in [0,t_k]$;
\item[(iv)] $\OB(g_k(t))\geq -1$ on $B_{g_k(t)}(p_k,1)$, $t\in [0,t_k]$;
\item[(v)]$|\Rm(g_k(t))|<L_k t^{-1}$ on $B_{g_k(t)}(p_k,\frac{1}{2})$, $t\in (0,t_k)$;
\item[(vi)]$|\Rm(g_k(z_k,t_k))|=L_k t_k^{-1}$ for some $z_k\in \overline{B_{g_k(t_k)}(p_k,\frac{1}{2})}$.
\end{enumerate}

By \cite[Lemma 2.3]{SimonTopping2016}, we may adjust $T_k$ (depending also on $L_k$) such that in addition we could have
\begin{equation}
V_{g_k(t)}(p_k,1)\geq v_1(n,v)>0
\end{equation}
for $t\in [0,T_k]$. By (vi) and $L_kt_k\rightarrow 0$, \cite[Lemma 5.1]{SimonTopping2016} implies that for $k$ sufficiently large, we can find $\tilde t_k\in (0,t_k]$ and $x_k\in B_{g_k(\tilde t_k)}(p_k,\frac{3}{4}-\frac{1}{2}\b_n\sqrt{L_k \tilde t_k})$ such that
\begin{equation}\label{CURVATURE}
|\Rm(g_k(x,t))|\leq 4|\Rm(g_k(x_k,\tilde t_k)|=4Q_k
\end{equation}
whenever $d_{g_k(\tilde t)}(x,x_k)<\frac{1}{8}\b_n L_kQ_k^{-1/2}$ and $\tilde t_k-\frac{1}{8}L_kQ_k^{-1}\leq t\leq \tilde t_k$ where $\tilde t_k Q_k\geq L_k\rightarrow +\infty$.

By Ricci lower bound of $g_k(t)$ and volume comparison, we may infer that for all $r\in (0,\frac{1}{4})$,
\begin{equation}\label{V}
\frac{V_{g_k(\tilde t_k)}(x_k,r)}{r^{2n}}\geq v_2(n,v)>0.
\end{equation}

Consider the parabolic rescaling centred at $(x_k,\tilde t_k)$, namely $\tilde g_k(t)=Q_kg(\tilde t_k+Q_k^{-1}t)$ for $t\in [-\frac{1}{8}L_k,0]$ so that $|\Rm_{\tilde g_k(0)}(x_k)|=1$. By \eqref{V} and the result of Cheeger-Gromov-Taylor \cite{CheegerGromovTaylor1982}, the injectivity radius of $\tilde g_k(0)$ at $x_k$ is bounded from below uniformly. Together with the curvature estimates inherited from \eqref{CURVATURE}, we may apply Hamilton's compactness \cite{Hamilton1995} so that $(M_k,\tilde g_k(t),x_k)$ converges in the Cheeger-Gromov sense to $(M_\infty,g_\infty(t),x_\infty)$ which is a complete non-flat ancient solution to the \KR flow with bounded curvature. By $(iii)$ and $(iv)$, $g_\infty(t)$ has nonnegative orthogonal bisectional curvature and Ricci curvature. By \eqref{V}, $g_\infty(t)$ is also of maximal volume growth which contradicts with \cite[Proposition 6.1]{NiLi2019}. This proves the curvature estimates.

The higher order estimate on ball of smaller radius follows from Shi's higher order local estimate \cite{Shi1989} by choosing a larger $L$, see \cite[Theorem 1.4]{CaoChenZhu2008} for the version that we used. By \cite[Lemma 2.3]{SimonTopping2016} (see also \cite[Corollary 6.2]{Simon2012}) and further shrinking $T_3$, we have $V_{g(t)}(p,1)\geq v_1(n,v)$ for $t\in [0,T_3]$. By volume comparison, we may further conclude that $V_{g(t)}(x,r)\geq v_2(n,v)r^{2n}$ for all $x\in B_{g(t)}(p,\frac{1}{4}),t\in [0,T_3]$ and $r\in (0,\frac{1}{4}]$. The injectivity radius lower bound can be proved using curvature upper bound and the result of Cheeger-Gromov-Taylor \cite{CheegerGromovTaylor1982}.
\end{proof}

\section{Existence of \KR flow}\label{s-existence}

\subsection{Existence of \KR flow on $B_{g_0}(R)\times[0,T]$}

In this section, we will prove the following slightly more general pyramid extension of \KR flow which in turn implies Theorem \ref{KRF-exist}. This can be viewed as a \K  analogy to \cite[Lemma 4.1]{McLeodTopping2019}, see also \cite[Lemma 2.1]{McLeodTopping2018} for three manifolds with Ricci curvature bounded from below. We proceed as in \cite{LeeTam2017}.

\begin{lma}\label{l-extension-2}
{\rm (Extension Lemma)} For all $\b_0\ge 1$ and $v>0$, there exist $a(n,v,\b_0)\geq 1,T(n,v,\b_0)>0$,  $\lambda(n,v,\beta_0)>0$ and   $\mu(n,v,\beta_0)>0$,  such that  the following is true:

 Suppose $(M^n,g_0)$ is a \K manifold with complex dimension $n$ and $p\in M$ so that $  B_{g_0}(p,R+1)\Subset M$ for some $R>0$ and for all $x\in B_{g_0}(p,R)$,
\begin{enumerate}
\item $\Ric(g_0)(x)\geq -\b_0$;
\item $\OB(g_0)(x)\geq -\b_0$;
\item $V_{g_0}(x,1)\geq v$.
\end{enumerate}
Suppose $g(t)$ is a   smooth \KR flow on $B_{g_0}(p,R)\times[0,t_0]$ with $0<t_0<T$ so that
\begin{enumerate}
\item[(i)] $|\Rm(g(t))|\leq at^{-1}$;
\item[(ii)] $\mathrm{inj}_{g(t)}(x)\geq \sqrt{a^{-1}t}$
\end{enumerate}
on $B_{g_0}(p,R)\times[0,t_0]$. Then $g(t)$ can be extended to a smooth solution to the \KR flow on $B_{g_0}(p, R-5\lambda t_0^\frac12)\times [0,(1+\mu)^2 t_0)]$ so that (i) and (ii) are still true, provided that $R-5\lambda t_0^\frac12>0$.

\end{lma}
\begin{proof}  Let $a(n,v,\beta_0)\ge 1, T(n,v,\beta_0)>0$, $\mu(n,v,\beta_0)>0$, $\lambda(n,v,\beta_0)>0$ to be determined.
In the following,

\begin{itemize}
  \item $c_i$ will denote positive constants depending only on $n$,  the lower bound of $v$ and upper bound of $\beta_0$; and
  \item $C_i$ will denote positive constants depending only on $n, a$.
\end{itemize}

By volume comparison, we have
\be\label{e-vol}
V_{g_0}(x,r)\ge c_1r^{2n}
\ee
for all $0<r\le 1$. Suppose $t_0<T$ and $R-5\lambda t_0^\frac12>0$.

 For $x\in B_{g_0}(p, R-\lambda t_0^\frac12)$, $B_{g_0}(x, \lambda t_0^\frac12)\Subset B_{g_0}(p,R)$. By Proposition \ref{OB-esti}, there exists $T_1(a,n)>0$ and $C_1(n,a)>1$ such that if $\lambda$  and $T$ satisfy
\be\label{e-condition-1}
(\lambda T^\frac12)^{-1}\ge \beta_0^\frac12
\ee
which implies $\mathrm{OB}(g_0)\ge -\beta_0\ge- (\lambda T^\frac12)^{-2}, \Ric(g_0) \ge\beta_0\ge- (\lambda T^\frac12)^{-2}$ in $B_{g_0}(x, \lambda t_0^\frac12)$, we have
\be\label{e-1st-1}
 \mathrm{OB}(g(x,t))\ge -C_1(\lambda^2 t_0)^{-1}; \Ric(g(x,t))\ge -C_1(\lambda^2 t_0)^{-1}
\ee
for all $t\le t_0\wedge (\lambda^2t_0 T_1)=t_0$, and $x\in B_{g_0}(p, R-2\lambda t_0^\frac12) $ provided
\be\label{e-condition-2}
\lambda^2 T_1\ge 1.
\ee
Also, by Lemma \ref{l-balls}, we have
\be\label{e-1st-2}
B_{g(t)}(x,\frac 12\lambda t_0^\frac12)\subset B_{g_0}(x,\lambda t_0^\frac12)
\ee
for $t\le t_0$, provided that
\be\label{e-condition-3}
\beta a^\frac12 \le \frac12\lambda.
\ee
where $\beta$ is a positive constant depending only on $n$. For $x\in B_{g_0}(p, R-3\lambda t_0^\frac12)$, $t\in [0,t_0]$,
$$
B_{g(t)}(x, \frac12 \lambda t_0^\frac12 C_1^{-\frac12})\subset B_{g(t)}(x,\frac 12\lambda t_0^\frac12)\subset B_{g_0}(x,\lambda t_0^\frac12)
$$
because $C_1>1$. Moreover,   $r=:  \frac12 \lambda t_0^\frac12 C_1^{-\frac12}\le1$ if
\be\label{e-condition-4}
\frac12\lambda T^\frac12\le 1.
\ee
By Proposition \ref{pseudo-OB}  and by \eqref{e-1st-1}, \eqref{e-vol}, there exist  $T_2(n,v,\beta_0)>0$ and $c_2(n,v,\b_0)$ such that
\begin{equation}\label{new-esti}
\left\{
\begin{array}{ll}
&|\Rm(x,t)|\leq c_2t^{-1};\\
&|\nabla \Rm(x,t)|\leq c_2^\frac32t^{-3/2};\\
&\mathrm{inj}_{g(t)}(x)\geq \sqrt{c_2^{-1}t}.
\end{array}\right.
\end{equation}
for all $ x\in B_{g_0}(p, R-3\lambda t_0^\frac12)$ and $t\le t_0\wedge r^2T_2=t_0$ provided
\bee
t_0\le r^2T_2=\frac14 \lambda^2 t_0  C_1^{-1}T_2,
\eee
or
\be\label{e-condition-5}
\lambda^2 T_2\ge 4C_1.
\ee
Let $U=B_{g_0}(p, R-3\lambda t_0^\frac12)$ and let $\rho=\sqrt{c_2^{-1}t_0}$. By Lemma \ref{l-extension-1}, we can find a solution to the \KR flow $h(s)$ defined on $U_\rho\times[0, \a c_2^{-1}t_0]$, where $\a=\a(n)$ is a positive constant depending only on $n$, and
$$
U_\rho=\{x:\ B_{g(t_0)}(x,\rho)\Subset B_{g_0}(p, R-3\lambda t_0^\frac12)\}
$$
with $h(0)=g(t_0)$ and
\be\label{e-h}
\a h(0)\le h(s)\le \a^{-1}h(0)
\ee
on $U_\rho$.
We claim
\be\label{e-U}
U_\rho\supset B_{g_0}(p,R-4\lambda t_0^\frac12).
\ee
In fact for $x\in B_{g_0}(p,R-4\lambda t_0^\frac12)$,
$$
B_{g(t_0)}(x,\rho)= B_{g(t_0)}(x,\sqrt{c_2^{-1}t_0})\Subset B_{g(t_0)}(x,\frac12\lambda t_0^\frac12)
$$
provided
\be\label{e-condition-9}
c_2^{-\frac12}< \frac12\lambda.
\ee
By \eqref{e-1st-2}, we conclude that the claim is true.

  By Lemma  \ref{l-curv1}, by \eqref{e-U}, \eqref{e-1st-2} and the definition of $U_\rho$, there exists $A=A(n)$ such that for $x\in B_{g_0}(p,R-5\lambda t_0^\frac12)$,
$$
|\Rm(h(x,s))|\le A\rho^{-2}=\frac{c_2A}{t_0}\le \frac{c_2A(1+\mu)^2}{(1+\mu)^2 t_0}\le \frac{c_2A(1+\mu)}{t}
$$
for all $t\le (1+\mu)^2t_0$, provided
\be\label{e-condition-8}
\mu\le \a c_2^{-1}.
\ee
We will choose $\mu\le 1$.  Hence we can extend $g(t)$ on $B_{g_0}(p,R-5\lambda t_0^\frac12)$ to a \KR flow defined on $[0,(1+\mu)t_0]$ if we define $g(t)=h(s+t_0)$ for $t\ge t_0$. Moreover the curvature of $g(t)$ satisfies
\bee
|\Rm(g(x,t))|\le \frac at
\eee
on $B_{g_0}(p,R-5\lambda t_0^\frac12)\times[0,(1+\mu)^2t_0]$
provided that,
\be\label{e-condition-6}
c_2A(1+\mu)^2\le a
\ee
By \eqref{new-esti} and \eqref{e-h}, we conclude that
the injectivity radius of $h(s)$ at $x\in B_{g_0}(p,R-5\lambda t_0^\frac12)$,
\bee
\mathrm{inj}_{h(s)}(x)\ge \sqrt{c_3^{-1}t_0}
\eee
for some constant $c_3$ depending only on $c_2, \a, n$ which implies that $c_3$ depends only on $n, v,\beta_0$ by \cite{CheegerGromovTaylor1982}. Hence
$
\mathrm{inj}_{g(t)}(x)\ge \sqrt{a^{-1}t}
$
for $t\ge t_0$ provided,
\be\label{e-condition-7}
a\ge c_3.
\ee
Hence the lemma is true, provided $a, T, \mu, \lambda$ can be chosen so that conditions \eqref{e-condition-1}, \eqref{e-condition-2}, \eqref{e-condition-3},
\eqref{e-condition-4}, \eqref{e-condition-5}, \eqref{e-condition-8},  \eqref{e-condition-6}, \eqref{e-condition-7} are satisfied. Let us list the conditions below:

\bee
\left\{
  \begin{array}{ll}
\mu\le \a_n c_2^{-1}, \mu\le 1, \ c_2=c_2(n,v,\beta_0); \\
a\ge c_3; \ c_3=c_3(n,v,\beta_0); \\
c_2A(1+\mu)^2\le a,\ c_2=c_2(n,v,\beta_0), A=A(n); \\
\lambda^2 T_1\ge 1,   T_1=T_1(n,a) ; \\
\beta a^\frac12 \le \frac12\lambda,\ \beta=\beta(n);\\
\lambda^2 T_2\ge 4C_1;\ T_2=T_2(n,v,\beta_0), \ C_1=C_1(n,a); \\
    (\lambda T^\frac12)^{-1}\ge \beta_0^\frac12; \\
    \frac12\lambda T^\frac12\le 1;\\
  \end{array}
\right.
\eee
 Choose $\mu(n,v,\beta_0)>0$ small enough so that the first inequality is true.
Now choose $a=a(n,v,\beta_0)>1$ large enough so that $a\ge c_3$ and $a\ge 4c_2A$, then the second and the third  inequalities are true.    \eqref{e-condition-8} and \eqref{e-condition-6} are satisfied. Then one can choose $\lambda(n,v,\beta_0)>0$ large enough so that the fourth, fifth and sixth inequalities are true. Finally, Choose $T(n,v,\beta_0)>0$ so that the last two inequalities are true. This completes the proof of the lemma.

\end{proof}

\begin{lma}\label{l-extension-3} For all $\b_0\ge 1$ and $v>0$, there exist $T(n,v,\b_0,B)>0$ and $L(n,v,\b_0)>0$  such that  the following is true:

 Suppose $(M^n,g_0)$ is a \K manifold with complex dimension $n$ and $p\in M$ so that $B_{g_0}(p,R)\Subset M$ for some $R>1$ and for all $x\in B_{g_0}(p,R)$,
\begin{enumerate}
\item $\Ric(g_0)(x)\geq -\b_0$;
\item $\OB(g_0)(x)\geq -\b_0$;
\item $V_{g_0}(x,r)\geq vr^{2n}$, for $r\le 1$ and $B_{g_0}(x,r)\subset B_{g_0}(p,R)$.
\end{enumerate}
Suppose $g(t)$ is a   smooth \KR flow on $B_{g_0}(p,R)\times[0,S]$
\begin{enumerate}
\item[(i)] $|\Rm(g(t))|\leq Bt^{-1}$;
\item[(ii)] $\mathrm{inj}_{g(t)}(x)\geq \sqrt{B^{-1}t}$
\end{enumerate}
Then we also have
\begin{enumerate}
\item[(i)] $|\Rm(g(t))|\leq Lt^{-1}$;
\item[(ii)] $\mathrm{inj}_{g(t)}(x)\geq \sqrt{L^{-1}t}$
\end{enumerate}
in $B_{g_0}(p,R-\frac{1}{2})$ and $0\le t\le S\wedge T$.
 \end{lma}
\begin{proof} We may assume that $\beta_0>100$.   By Proposition \ref{OB-esti}, there exist $C_1=C_1(n,B)>1$, $T_1=T_1(n,B)>0$ with $r=(\sqrt{\b_0})^{-1}$ such that
$$
\OB(  g(x,t))\geq -    C_1\beta_0; \Ric(g(x,t))\ge -C_1\beta_0
$$
for all $x\in B_{g_0}(p,R-2\beta_0^{-\frac12})$ for all $0\le t\le S\wedge \beta_0^{-1}T_1$. Here we have used the fact that $r<\frac13$.

   By  Proposition \ref{pseudo-OB} and Lemma \ref{l-balls}, there exist $L(n,v)>0$, $T_2(n,v,B)>0$ for all $x\in B_{g_0}(p,R-(2+C_1^{-\frac12})\beta_0^{-\frac12})$ so that $B_{g_0}(x,  (C_1\beta_0)^{-\frac12})\Subset M$, we have $B_{g(t)}(x, \frac{1}{2}(C_1\beta_0)^{-\frac12})$ for $t\in S\wedge (C_1\b_0)^{-1}T_2$ and hence
 $$
 |\Rm(g(x,t))|\le \frac{L}t; \mathrm{inj}_{g(t)}(x)\ge \sqrt{L^{-1}t}
 $$
 for all $0\le t\le S\wedge \beta_0^{-1}T_1\wedge (C_1\beta_0)^{-1}T_2$. Since $\beta_0>100$, $C_1>1$, we have $(2+C_1^{-\frac12})\beta_0^{-\frac12}<\frac{1}{2}$. From this it is easy the lemma is true.

\end{proof}
\begin{rem}\label{r-a} Since the constant $L(n,v)$ in Lemma \ref{l-extension-3} depends only on $v$, $n$, by the proof of Lemma  \ref{l-extension-2}, we may assume that $a$ in the Lemma \ref{l-extension-2} also satisfies $a\ge L$.
\end{rem}

{ We are ready to prove Theorem \ref{t-intro-local-ext-KRF}}.

\begin{proof}[Proof of Theorem \ref{t-intro-local-ext-KRF}] Let $a, T, \lambda, \mu$ as in Lemma \ref{l-extension-2}. By Lemma \ref{l-extension-3} and Remark \ref{r-a}, there exists $T_1(n,v,\beta_0,B)>0$ such that
\be\label{e-extension-1}
\left\{
  \begin{array}{ll}
     |\Rm(g(t))|\leq at^{-1}; \\
    \mathrm{inj}_{g(t)}(x)\geq \sqrt{a^{-1}t}
  \end{array}
\right.
\ee
in $B_{g_0}(p, R+\frac12)\times[0, S\wedge T_1]\supset B_{g_0}(p, R+\frac12)\times[0, S\wedge T_0]$, where  $T_0=T_1\wedge T$. Suppose $S\ge T_0$,  then the theorem is obviously true. Suppose $S<T_0$. Let
$$
R_0=R+\frac12, t_0=S
$$
and let
$$
t_{k+1}=(1+\mu)^2t_{k}, R_{k+1}=R_{k}-5\lambda t_{k}^\frac12
$$
for $k\ge 0$. For $k=0$, then  $R_0=R+\frac12>R$ and $t_0=S<T_0$. Since $t_k\uparrow\infty$, there is  $k_0\ge1$ such that $t_k< T_0$ and $R_k>R$ for all $k<k_0$. Moreover, $t_{k_0}\ge T_0$ or $R_{k_0}\le R$. Clearly, $k_0\ge 1$. By Lemma \ref{l-extension-2}, we conclude that $g(t)$ can be extended to  $B_{g_0}(p, R_{k_0-1})\times[0, t_{k_0-1}]$ satisfying the conditions (1) and (2).

Suppose $R_{k_0}\le R$, then
\bee
\begin{split}
\frac12\le &\sum_{k=0}^{k_0-1}5\lambda t_k^\frac12\\
=&5\lambda t_{k_0}^\frac12\sum_{i=1}^{k_0}(1+\mu)^{-i}\\
\le &5\lambda\mu^{-1} t_{k_0}^\frac12.
\end{split}
\eee
So
\bee
t_{k_0}\ge \left(\frac1{10}\mu\lambda^{-1}\right)^2.
\eee
and
\bee
t_{k_0-1}= (1+\mu)^{-2}t_{k_0}\ge (1+\mu)^{-2}(\frac1{10}\mu\lambda^{-1})^2\ge T_0
\eee
if $T_0$ is adjusted so that $T_0\le  (1+\mu)^{-2}(\frac1{10}\mu\lambda^{-1})^2$. Since $R_{k_0-1}>R$, we conclude that $g(t)$ can be extended to $[0,T_0]$ on $B_{g_0}(p,R)$ satisfying (1) and (2).

  If $R_{k_0}>R$ and $t_{k_0}\ge T_0$.  Then $R_{k_0-1}>R$ and $t_{k_0-1}< T_0$ by the definition of $k_0$. By Lemma \ref{l-extension-2}, $g(t)$ can be extended to $[0,t_{k_0}]\supset [0,T_0]$ on $B_{g_0}(p, R_{k_0})\supset B_{g_0}(p,R)$ satisfying (1) and (2). This completes the proof of the theorem.

\end{proof}

\subsection{Existence of \KR flow on $M\times[0,T]$}

The following is a consequence of Theorem \ref{t-intro-local-ext-KRF}:

\begin{lma}\label{l-local}
For any $\beta_0\ge 1$, $v>0$, there exist $a(n,v,\beta_0)\ge 1$, $T(n,v,\beta_0)>0$, and $C(n,v,\beta_0)\ge0$ such that the following is true:

Suppose $(M^n,g_0)$ is a \K manifold with complex dimension $n$. Let $p\in M$ so that $B_{g_0}(p,R+4)\Subset M$ for some $R> 1$ and for all $x\in B_{g_0}(p,R+3)$,
\begin{enumerate}
\item [(i)]$\Ric(g_0)\geq -\b_0$;
\item [(ii)]$\OB(g_0)\geq -\b_0$;
\item [(iii)] $V_{g_0}(x,1)\geq v$.
\end{enumerate}
Then  there is a smooth \KR flow $g(t)$ solution on $B_{g_0}(p,R)\times [0,T]$ such that
\begin{enumerate}
\item $|\Rm(g(t))|\leq at^{-1}$;
\item $\mathrm{inj}_{g(t)}(x)\geq \sqrt{a^{-1}t}$; and
\item $\OB(g(x,t)\geq -C(n,v,\beta_0)$; $\Ric(g(x,t)\geq -C(n,v,\beta_0)$.
\end{enumerate}
 \begin{proof}  Let $a(n,v,\b_0)$ be the constant obtained from Theorem \ref{t-intro-local-ext-KRF}.
On $B_{g_0}(p,R+4)$, choose $1>>\rho>0$ small enough so that for all $x\in B_{g_0}(p,R+3)$,
\begin{equation}
\left\{
\begin{array}{ll}
& B_{g_0}(x,\rho)\Subset M;\\
&|\Rm (g_0)|\leq \rho^{-2} ;\\
&\mathrm{inj}_{g_0}(x) \geq \rho.
\end{array}
\right.
\end{equation}
Then we may apple Lemma \ref{l-extension-1} with $N=M$, $U=B_{g_0}(p,R+3)$ to get a solution to the \KR flow $\wt  g(t)$ with $ \wt g(0)=g_0$ defined on $B_{g_0}(p,R+2)\times [0,\a_n\rho^2]$. By smoothness of $ \wt  g(t)$, we may choose  $\rho$ small enough so that
so that for all $(x,t)\in B_{g_0}(p,R+2)\times [0,\a_n\rho^2]$,
\begin{equation}
\left\{
\begin{array}{ll}
&|\Rm(\wt g(t))|\leq at^{-1};\\
&\mathrm{inj}_{\wt g(t)}(x)\geq \sqrt{a^{-1}t}.
\end{array}\right.
\end{equation}

 By   Theorem \ref{t-intro-local-ext-KRF}, $\wt g(t)$ can be extended to a solution to the \KR flow on $B_{g_0}(p, R+1)\times[0, T]$ for some $T(n,v,\beta_0)>0$ so that $g(t)$ satisfies (1) and (2) in the lemma. Property (3) in the lemma follows from Proposition \ref{OB-esti} by choosing a possible smaller $T>0$ which depends only on $n, v,\beta_0$, since  $a$ depends only on $n ,v,\beta_0$.

 \end{proof}
\end{lma}

Now Theorem \ref{t-intro-pyramid} (\textbf{II}) follows from Lemma \ref{l-local} using exhaustion argument. We restate it for reader's convenience.
\begin{thm}\label{KRF-exist}
For any $n\in \mathbb{N},v>0$, there exist $T(n,v), a(n,v), L(n,v)>0$ such that the following holds. Suppose $(M,g_0)$ is a complete noncompact \K manifold with
\begin{enumerate}
\item[(a)] $\Ric(g_0)\geq -1$ on $M$;
\item[(b)] $\OB(g_0)\geq -1$ on $M$;
\item[(c)] $V_{g_0}(x,1)\geq v$ for all $x\in M$,
\end{enumerate}
then there is a complete solution $g(t)$ to the \KR flow starting from $g_0$ such that
\begin{equation}
\left\{
\begin{array}{ll}
\Ric(g(t))&\geq -L;\\
\OB(g(t))&\geq -L;\\
|\Rm(g(t))|&\leq at^{-1};\\
\mathrm{inj}_{g(t)}(x)&\geq \sqrt{a^{-1}t}
\end{array}
\right.
\end{equation}
on $M\times (0,T]$.
\end{thm}
\begin{proof}
Fix $p\in M$ and denote $\Omega_i=B_{g_0}(p,i)$, $i\in \mathbb{N}$ with $i\ge 5$.
By Lemma \ref{l-local}, there exist   $T(n,v,\b_0)>0, a(n,v,\b_0)>0, C(n,v,\beta_0)>0$ such that for all $i\in \mathbb{N}$, we can find a  \KR flow $g_i(t), t\in [0,T]$ defined on each $\Omega_i$ which satisfies
\begin{equation}\label{approx-solut}
\left\{
  \begin{array}{ll}
      &\Ric(g_i(x,t)\geq -C(n,v,\beta_0);\\
      &\OB(g_i(x,t)\geq -C(n,v,\beta_0);\\
      &|\mathrm{\Rm}(g_i(t))|\leq {a}t^{-1}   \\
      & \mathrm{inj}_{g_i(t)}(x)\geq \sqrt{a^{-1}t}.
  \end{array}
\right.
\end{equation}
for $(x,t)\in \Omega_i\times[0,T]$.

By \cite[Corollary 3.2]{Chen2009} (see also \cite{Simon2008}) and the modified Shi's higher order estimate \cite[Theorem 14.16]{ChowBookII}, we infer that for any $j,k\in \mathbb{N}$, we can find $C(n,k,\Omega_j,g_0,v,\b_0)>0$ so that for all $i>j$,
\begin{align}
\sup_{\Omega_{j-1}\times [0,T]}|\nabla^k \mathrm{Rm}(g_i(t))|\leq C(n,k,\Omega_j,g_0,v,\b_0).
\end{align}

By working on coordinate charts and  Ascoli-Arzel\`a Theorem, we may pass to a subsequence to obtain a smooth solution $g(t)=\lim_{i\rightarrow +\infty}g_i(t)$ of the K\"ahler-Ricci flow on $M\times [0,T]$ with $g(0)=g_0$ so that $|\mathrm{Rm}|\leq at^{-1}$ on $M\times (0,T]$ and
\begin{equation}
\left\{
\begin{array}{ll}
&\OB(g(x,t))\geq -C(n,v,\beta_0);\\
&\Ric(g(x,t))\geq -C(n,v,\beta_0).
\end{array}
\right.
\end{equation}
 for all $(x,t)\in M\times[0,T]$.
 Moreover, it is a complete solution by Lemma \ref{l-balls}.   This completes the proof of the theorem.
\end{proof}

\subsection{Partial \KR flow}

Next we want to prove the existence of  partial \KR flow, namely  Theorem  \ref{t-intro-pyramid} {\bf (I)}  which will be used  to study Gromov Hausdorff limit of complete \K manifolds with almost non-negative curvature as in  \cite{McLeodTopping2018,McLeodTopping2019}. We prove a more general version:

\begin{thm}\label{pyramid-exten-lma}
For any $n, v_0>0$ and any nondecreasing positive function $f(r):[0,\infty)\to (1,\infty)$, there exist  nondecreasing sequence   $a_k,\b_k\geq 1$ and nonincreasing   sequence $S_k>0$ such that the following holds:

Suppose $(M^n,g_0)$ is a complete non-compact \K manifold and $p\in M$ so that
\begin{enumerate}
\item $\Ric(g_0)\geq -f(r)$ on $B_{g_0}(p,r)$ for all $r>0$;
\item $\OB(g_0)\geq -f(r)$ on $B_{g_0}(p,r)$ for all $r>0$;
\item $V_{g_0}(p,1)\geq v_0$.
\end{enumerate}
Then for any $m\in \mathbb{N}$, there is a solution to the \KR flow $g_m(t)$ defined on a subset $D_m$ of spacetime given by
$$D_m=\bigcup_{k=1}^m \left(B_{g_0}(p,k)\times [0,S_k]\right),$$
with $g_m(0)=g_0$ on where it is defined and satisfies
\begin{equation}\label{py-esti}
\left\{
\begin{array}{ll}
&\Ric(g_m(t))\geq -\b_k;\\
&\OB(g_m(t))\ge -\b_k;\\
&|\Rm(g_m(t))|\leq a_kt^{-1};\\
&\mathrm{inj}_{g_m(t)}(x)\geq \sqrt{a_k^{-1}t}
\end{array}
\right.
\end{equation}
on each $B_{g_0}(p,k)\times (0,S_k]$.
\end{thm}
\begin{proof}
The proof is similar to that in \cite[Theorem 1.2]{McLeodTopping2018}. For the sake of completeness, we sketch the proof here. By volume comparison, for $k\in \mathbb{N}$, there exists a sequence $v_k(n,v_0,f)>0$ such that for all $x\in B_{g_0}(p,k+4)$, $V_{g_0}(x,1)\geq v_k$.

{\bf Part A.} By Lemma \ref{l-local}, for each $k\ge 2$, there exist $a_k(n,v_k, f(k+4))>1, \beta_k(n,v_k,f(k+4))\ge 1$, $T_k(n,v_k, f(k+4))>0$ and smooth solution $g_k(t)$ defined on $B_{g_0}(p,k)\times[0,T_k]$ such that

\bee
\left\{
\begin{array}{ll}
&\Ric(g_k(t))\geq -\b_k;\\
&\OB(g_k(t))\ge -\b_k;\\
&|\Rm(g_k(t))|\leq a_kt^{-1};\\
&\mathrm{inj}_{g_k(t)}(x)\geq \sqrt{a_k^{-1}t}
\end{array}
\right.
\eee
We may assume that $a_k$ is nondecreasing.

{\bf Part B.} By Theorem \ref{t-intro-local-ext-KRF}, Lemma \ref{l-extension-3} and Proposition \ref{OB-esti}, for each $k$ there is $\wt T_k(n, v_{k+1}, a_k, a_{k+1})>0 $ such that for any smooth solution $h(t)$ defined on   $B_{g_0}(p, k+1)\times[0,T]$ with $h(0)=g_0$ so that
\bee
\left\{
\begin{array}{ll}
&\Ric(h(t))\geq -\b_{k+1};\\
&\OB(h(t))\ge -\b_{k+1};\\
&|\Rm(h(t))|\leq a_{k+1}t^{-1};\\
&\mathrm{inj}_{h(t)}(x)\geq \sqrt{a_{k+1}^{-1}t}
\end{array}
\right.
\eee
can be extended to a smooth solution $\tilde h(t)$ of the \KR flow on  $B_{g_0}(p,k)\times[0,    \wt T_k]$ so that $\tilde h(t)=h(t)$ on $[0,\wt T_k\wedge T]$ and satisfies
\bee
\left\{
\begin{array}{ll}
&\Ric(h(t))\geq -\b_{k};\\
&\OB(h(t))\ge -\b_{k};\\
&|\Rm(h(t))|\leq a_{k}t^{-1};\\
&\mathrm{inj}_{h(t)}(x)\geq \sqrt{a_{k}^{-1}t}
\end{array}
\right.
\eee
on $B_{g_0}(p,k)\times [0,\tilde T_k]$ by choosing larger $\beta_k$ and $a_k$. We may adjust so that $\tilde T_k$ is non-increasing.

 Define $S_k=\tilde T_k$ for $k\in \mathbb{N}$. Now fixed $m\in \mathbb{N}$ with $m \ge 2$. By {\bf Part A}, we can find a solution $g_m(t)$ to the \KR flow defined on $B_{g_0}(p,m+1)\times [0,T_{m+1}]$. Then using {\bf Part B}, $g_m(t)$ admits a local extension on $B_{g_0}(p, m)\times[0,S_m]$ so that
\bee
\left\{
\begin{array}{ll}
&\Ric(g_{m}(t))\geq -\b_{m};\\
&\OB(g_m(t))\ge -\b_{m};\\
&|\Rm(g_m(t))|\leq a_{m}t^{-1};\\
&\mathrm{inj}_{g_m(t)}(x)\geq \sqrt{a_{m}^{-1}t}
\end{array}
\right.
\eee
By the choice of $S_k$, $g_m(t)$ can be extended to $[0,S_{m-1}]$ on $B_{g_0}(p,m-1)$, still denoted by $g_m(t)$, so that
\bee
\left\{
\begin{array}{ll}
&\Ric(g_{m}(t))\geq -\b_{m-1};\\
&\OB(g_m(t))\ge -\b_{m-1};\\
&|\Rm(g_m(t))|\leq a_{m-1}t^{-1};\\
&\mathrm{inj}_{g_m(t)}(x)\geq \sqrt{a_{m-1}^{-1}t}
\end{array}
\right.
\eee
on $B_{g_0}(p,m-1)\times[0,S_{m-1}]$. Inductively, $g_m(t)$ can be further extended to $[0,S_{k}]$ on each $B_{g_0}(p,k)$ for $m>k\ge2$  so that \eqref{py-esti} is true. This completes the proof of the theorem.
\end{proof}

\section{Applications}

\subsection{Gromov-Hausdorff limit of \K manifolds}

The first application is to use the \KR flow to smooth a metric space which is the limit of a sequence of complete strongly or weakly non-collapsing \K manifolds with the almost non-negative curvature conditions. Namely, we obtain Theorem \ref{t-intro-GH}. The proof of part (\textbf{I}) is more tedious, but the idea is similar to the more easy proof of part (\textbf{II}). Hence we begin to prove this part to illustrate the idea.

\begin{proof}
[Proof of Theorem \ref{t-intro-GH} (\textbf{II})]
By using the global existence result of the \KR flow, the proofs here follow almost verbatim from the arguments in \cite[Corollary 4]{BCRW}, \cite[Corollary 1.3]{Lai2019} and \cite[Theorem 1.8]{SimonTopping2017}.
By Theorem \ref{KRF-exist}, there is uniform constant $T,a$ and $L$ depending only on $n$ and $v$ such that for each $i$, there is a short-time solution $g_i(t)$ to the \KR flow defined on $M_i\times [0,T]$ with
\begin{enumerate}
\item  $\Ric(g_i(t))\geq -L$;
\item  $|\Rm(g_i(t))|\leq at^{-1}$;
\item $\text{inj}_{g_i(t)}(x)\geq \sqrt{a^{-1}t}$
\end{enumerate}
 on $M_i\times (0,T]$. By Lemma \ref{Lemma3.1ST}, for all $0\leq s<t<T$, $i\in \mathbb{N}$ and for all $x,y\in M_i$,
\begin{align}\label{dist-seq}
d_{g_i(s)}(x,y)-\b_n \sqrt{a}(\sqrt{t}-\sqrt{s})\leq d_{g_i(t)}(x,y)\leq e^{L(t-s)}d_{g_i(s)}(x,y).
\end{align}

On the other hand by Hamilton's compactness \cite{Hamilton1995}, we may pass it to subsequence so that $(M_i,g_i(t),p_i)\rightarrow (M_\infty,g_\infty(t),p_\infty)$ in the Cheeger-Gromov sense for $t\in (0,T]$. In particular, $g_\infty(t)$ is a complete \KR flow solution on $(0,T]$ which satisfies
\begin{align}\label{dit}
d_{g_\infty(s)}(x,y)-\b_n \sqrt{a}(\sqrt{t}-\sqrt{s})\leq d_{g_\infty(t)}(x,y)\leq e^{L(t-s)}d_{g_\infty(s)}(x,y)
\end{align}
for all $0<s<t<T$ and $x,y\in M_\infty$. Hence $d_\infty(x,y)=\lim_{s\rightarrow 0}d_{g_\infty(s)}(x,y)$ exists as a distance function.   The fact that $d_\infty$ and $d_{g_\infty(t)}$ generate the same topology as the one on $M_\infty$ follows from Lemma \ref{Lemma3.1ST}.

  It remains to show that $(M_i,d_{g_i},p_i)$ converges to $(M_\infty,d_\infty, p_\infty)$ in the pointed Gromov-Hausdorff sense. Here $g_i=g_i(0)$. By \cite[Definition 8.1.1]{BYS},    it suffices to prove the following.

\begin{claim}
For all $r>0,\e>0$, there exists $N$ such that for all $i>N$, we can find $f_i:B_{g_i}(p_i,r)\rightarrow M_\infty$ with
\bee\left\{
      \begin{array}{ll}
        f_i(p_i)=p_\infty; \\
        |d_{g_i}(x,y)-d_{\infty}(f_i(x),f_i(y))|<\e, \ \text{for all $x,y\in B_{g_i}(p_i,r)$}; \\
      B_{d_\infty}(p_\infty,r-\e)  \text{\ is a subset $\e$-neighborhood of $f_i(B_{g_i}(p_i,r))$ with respect to $d_\infty$}.
      \end{array}
    \right.
\eee
\end{claim}
Here $g_i=g_i(0)$.

To prove the claim, let $r>0$. Consider the identity embedding, $\phi_i: B_{g_i}(p_i, r)\to M_i$. By \eqref{dist-seq}, we have
\be\label{e-GH-1}
 \left\{
    \begin{array}{ll}
      \phi(p_i)=p_i; \\
      |d_{g_i}(x,y)-d_{g_i(t)}(\phi_i(x),\phi_i(y))|\le C_1t^\frac12(1+r), \ \text{for all $x,y\in B_{g_i}(p_i,r)$};   \\
        B_{g_i(t)}(p_i,r-C_1t^\frac12)\subset\phi_i( B_{g_i}(p_i,r))\subset B_{g_i(t)}(p_i, (1+C_1 t)r).
    \end{array}
  \right.
\ee
for $0<t<\min\{1,T\}$, for some constant $C_1$ independent of $i, t$ and $r$.
For fixed $t$,  $(M_i,g_i(t),p_i)$ converge to  $(M_\infty,g_\infty(t),p_\infty)$ in the Cheeger-Gromov sense and hence in the PGH sense, see \cite[Lemma 6.1]{SimonTopping2017}. This implies for any $\lambda>0$, $\delta>0$ there is $i_0$ such that for any $i\ge i_0$, there exists diffeomorphism $\theta_i: B_{g_i(t)}(p_i, \lambda)\to M_\infty$ satisfying the following for all $0<\mu<\lambda$.
\be\label{e-GH-2}
 \left\{
    \begin{array}{ll}
      \theta_i(p_i)=p_\infty; \\
      |d_{g_i(t)}(x,y)-d_{g_\infty(t)} (\theta_i (x),\theta_i(y))|<\delta, \ \text{for all $x,y\in B_{g_i(t)}(p_i,\lambda)$};   \\
        B_{g_\infty(t)}(p_\infty,\mu-\delta)\subset \theta_i(B_{g_i(t)}(p_i, \mu) \subset B_{g_\infty(t)}(p_\infty,\mu+\delta).
    \end{array}
  \right.
\ee
Here $i_0$ depends on $\lambda,\delta, t$. In fact the

By \eqref{dit},  for any $\rho>0$, let $\sigma: B_{g_\infty(t)}(p_\infty, \rho)\to M_\infty$ be the identity embedding $\sigma: B_{g_\infty(t)}(p_\infty,\rho)\to M_\infty$, then we have

\be\label{e-GH-3}
 \left\{
    \begin{array}{ll}
      \sigma(p_i)=p_i; \\
      |d_{g_\infty(t)}(x,y)-d_\infty (\sigma (x),\sigma(y))|\le C_2t^\frac12(1+\rho), \ \text{for all $x,y\in B_{g_\infty(t)}(p_\infty,\rho)$};   \\
        B_{d_\infty}(p_\infty,(1-C_2t)\rho)\subset B_{g_\infty(t)}(p_\infty,\rho).
    \end{array}
  \right.
\ee

Now let $\lambda=(1+C_1)r$, $\rho=\lambda+\delta$. Let $f_i=\sigma\circ\theta_i\circ \phi_i$   which is a well-defined map from $B_{g_i}(p_i,r)$ into $M_\infty$ so that
$f_i(p_i)=p_\infty$. In fact, if $x\in B_{g_i}(p_i,r)$, then $\phi_i(x)\in B_{g_i(t)}(p_i, (1+C_1t)r)=B_{g_i(t)}(p_i, \lambda)$ and $\theta_i(\phi_i(x))\in B_{g_\infty(t)}(p_\infty, \lambda+\delta)=B_{g_\infty(t)}(p_\infty, \rho)$. Hence for all $x, y\in B_{g_i}(p_i,r)$,
\be\label{e-GH-4}
\begin{split}
 &|d_{g_i}(x,y)- d_\infty(f_i(x),f_i(y))|\\
&\le
  |d_{g_i}(x,y)-d_{g_i(t)}(\phi_i(x),\phi_i(y))|\\
  &+  |d_{g_i(t)}(\phi_i(x),\phi_i(y))-d_{g_\infty(t)}(\theta_i\circ\phi_i(x),\theta_i\circ\phi_i(y))|\\
&
+|d_{g_\infty(t)}(\theta_i\circ\phi_i(x),\theta_i\circ\phi_i(y))-d_\infty (\sigma\circ\theta_i\circ\phi_i(x),\sigma\circ\theta_i\circ\phi_i(y))|\\
&\le C_1t^\frac12(1+r)+\delta+C_2t^\frac12(1+\rho)\\
=&(C_1+C_2)t^\frac12( r+(2+C_1)r+2+\delta).
\end{split}
\ee

Next let $\mu=r-C_1t^\frac12<\lambda$, by \eqref{e-GH-1}--\eqref{e-GH-3}, we have
\be\label{e-GH-5}
\begin{split}
f_i(B_{g_i}(p_i,r))=&\sigma\circ\theta_i\circ\phi_i(B_{g_i}(p_i,r))\\
 \supset&\sigma\circ\theta_i (B_{g_i(t)}(p_i,\mu)) \\
\supset& \sigma(B_{g_\infty(t)}(p_\infty,\mu-\delta)\\
\supset& B_{d_\infty}(p_\infty,(1-C_2t) (\mu-\delta))
\end{split}
\ee
Note that
$$
(1-C_2t) (\mu-\delta) =(1-C_2t)(r-C_1t^\frac12-\delta)
$$
For any $r>0$ and any $\e>0$, choose $\delta>0$ small enough, and then choose $t$ small enough, one can see that the claim is true. This completes the proof of the theorem.
\end{proof}

Next we want  to use the pyramid \KR flow to consider weakly non-collapsing complete \K manifolds with almost non-negative curvature condition. The following theorem covers Theorem \ref{t-intro-GH}{\textbf (I)} which is based on the construction by McLeod-Topping in \cite{McLeodTopping2018,McLeodTopping2019} and the local \KR flow construction.

\begin{thm}\label{KRF-GHH}
Suppose that $(M_i,g_i,J_i,p_i)$ is a sequence of complete, sooth pointed \K manifolds such that for some $\b_0,v_0>0$, we have
$\OB(g_i) \geq -\b_0$, $\Ric(g_i)\geq -\b_0$ on $M_i$ and $V_{g_i}(p_i,1)\geq v_0$ for all $i\in M$.
Then there exists a smooth complex manifold $(M_\infty,J_\infty)$, a point $x_\infty\in M_\infty$ and a complete distance metric $d_\infty:M_\infty\times M_\infty\rightarrow [0,+\infty)$ generating the same topology as the one on $M_\infty$ and a smooth Ricci flow $g(t)$ defined on a subset of $M_\infty\times (0,+\infty)$ that contains $B_d(p_\infty,k)\times (0,T_k]$ for $k\in \mathbb{N}$ with $d_{g(t)}\rightarrow d_\infty$ locally uniformly on $M$ as $t\rightarrow 0$ and after passing to a subsequence in $i$, we have $(M_i,d_{g_i},p_i)$ converges to $(M_\infty,d_\infty,p_\infty)$ in the pointed Gromov Hausdorff sense. Moreover, $g(t)$ is \K with respect to the complex structure $J_\infty$ on where it is defined.
\end{thm}
\begin{proof}

The existence and convergence of $(M_\infty, d_\infty, p_\infty, g(t))$ as a smooth manifold and Riemannian Ricci flow was originated in \cite[Theorem 5.1]{McLeodTopping2018}. In fact, one can prove this using the argument is similar to the  proof of Theorem  \ref{t-intro-GH} (\textbf{II}) if we obtain relations similar to \eqref{e-GH-1}--\eqref{e-GH-3}. These will be accomplished as follows.

{\sc Step 1:}
By Theorem \ref{pyramid-exten-lma} and Ascoli-Arzel\`a Theorem, there exists a sequence of $T_l$ and $a_l$ such that for any $i\in \mathbb{N}$, we can find a pyramid \KR flow $g_i(t)$ with $g_i(0)=g_i$ defined on $\bigcup_{l=1}^\infty B_{g_i}(p_i,l+1)\times [0,T_l]$ and satisfies
\begin{enumerate}
\item [\bf (a1)] $\textbf{inj}_{g_k(t)}(x)\geq \sqrt{a_l^{-1}t}$;
\item [\bf (b1)] $|\Rm(g_i(x,t))|\leq a_lt^{-1}$;
\item [\bf (c1)] $\Ric(g_i(x,t))\geq -a_l$;
\item [\bf (d1)] $\OB(g_i(x,t))\geq -a_l$
\end{enumerate}
for $(x,t)\in B_{g_i}(p_i,l+1)\times (0,T_l]$.   By shrinking balls Lemma \ref{l-balls}, {\textbf{(c1)}} and Lemma \ref{Lemma3.1ST}, one can see that for any $\e>0$, one may choose $T_l$ small enough which is independent of $i$ so that
$$
B_{g_i}(p_i, l+1-2\e)\subset B_{g_i(t)}(p_i, l+1-\e)\subset B_{g_i}(p_i, l+1);
$$
and  for $0<s<t<T_l$ and $x, y\in  B_{g_i}(p_i,\frac l2 l)$ there is $c_l>0$ independent of $i$ with

\be\label{e-GH-partial-1}
d_{g_i(s)}(x,y)-c_l(\sqrt t-\sqrt s)\le d_{g_i(t)}(x,y)\le e^{c_l(t-1)}d_{g_i(s)}(x,y).
\ee

 {\sc Step 2}: By the local Hamilton's compactness \cite[Lemma B.3]{McLeodTopping2018}, for each $l\ge 2$, there is a smooth pointed Riemannian  manifold $(N_l, h_l, q_l)$, a decreasing sequence $S_l>0$ with $T_l\ge S_l>0$ and a Ricci  flow $h_l(t)$  which may not be complete defined on   $N_l\times(0,S_l]$  with the following properties:

 \begin{itemize}
   \item[\bf (a2)] $B_{h_l(t)}(q_l,l+1)\Subset N_l$ for $t\in (0,S_l)$;
   \item [\bf (b2)]for $t\in (0,S_l)$, $B_{h_l(t)}(q_l, l)\subset \Omega_l$ which is the connected component of $q_l$ containing $q_l$ in
       $$
       \bigcap_{t\in (0,S_l)}B_{h_l(t)}(q_l, l+1);
       $$

   \item[\bf (c2)] passing to a subsequence, for each $i$, there is a smooth map
   $\phi_{l;i}$ from $\Omega_l$ into $B_{g_i(0)}(p_i,l+2)$ with $\phi_{l;i}(q_l)=p_i$, which is diffeomorphic onto its image such that $\phi_{l;i}^*(g_i(t))$ converge uniformly in $C^\infty$ norm in $\Omega_l\times[\delta,S_l]$ for all $\delta>0$.
 \end{itemize}

By a diagonal process, we may assume that the subsequence in $i$ is independent of $l$. Hence for each $l$, and for any $0<\delta<S_l$, if $i$ is large enough depending only on $l$ and $\delta$,
\begin{equation}\label{BB-in}
B_{g_i(t)}(p_i, l-\frac14)\subset\phi_{l;i}(\Omega_l)\subset \bigcap_{s\in (\delta,S_l)}B_{g_i (s)}(p_i, l+\frac54).
\end{equation}
In particular, $\phi_{l;i}(\Omega_l)\subset \phi_{l+1;i}(\Omega_{l+1})$. Hence
$$
\phi_{l+1;i}^{-1} \circ\phi_{l;i}:\Omega_l\to \Omega_{l+1}
$$
is a well-defined  diffeomorphism onto its image, which maps $q_l$ to $q_{l+1}$. Combining these  with {\bf (c2)}, let $i\to\infty$, we conclude $\phi_{l+1;i}^{-1} \circ\phi_{l;i}$ converge to a smooth map
$$
\phi_l:\Omega_l\to \Omega_{l+1}
$$
which is a diffeomorphism onto its image, with $\phi_l(q_l)=q_{l+1}$ such that $$\phi_l^*(h_{l+1}(t))=h_l(t)$$ for $0<t<S_{l+1}$ on $\phi_l(\Omega_l)$. Hence one can glue $\Omega_l$ to be a smooth manifold. More precisely, by \cite[Theorem C.1]{McLeodTopping2018}, there exists a smooth manifold $M_\infty$, a point $p_\infty\in M_\infty$ smooth maps diffeomorphisms
$$
F_l:\Omega_l\to M_\infty
$$
which is diffeomorphic onto its image, so that $F_l(q_l)=p_\infty$ and
\begin{itemize}
  \item [{\bf (a3)}] $M_\infty=\bigcup_{l=2}^\infty F_l(\Omega_l)$;
  \item [{\bf (b3)}] $F_l(\Omega_l)\subset F_{l+1}(\Omega_{l+1})$;
  \item [{\bf (c3)}] $\phi_l=F_{l+1}^{-1}\circ F_l$.
\end{itemize}
 Since $\phi_l^*(h_{l+1}(t))=h_l(t)$, by the above, $(F_l)_* (h_l(t))$ can be glued together to a smooth $g_\infty(t)$ which satisfies the Ricci flow equation. Note that  on each $F_l(\Omega_l)$, $g_\infty(t)$ is defined on $(0,S_{l})$ and $M_\infty$ is complete by {\bf (b2)}. Moreover, in $F_l(\Omega_l)\times(0,S_l)$, we have

  \begin{itemize}
       \item [{\bf (a4)}] $\Ric(g_\infty(t))\ge -a_l$;
       \item[{\bf (b4)}] $|\Rm(g_\infty(t)|\le a_l/t$;
       \item[{\bf (c4)}] $\OB(g_\infty(t))\geq -a_l$.
     \end{itemize}

  Let $f_i^l=\phi_{l;i}\circ F_l^{-1}$. Then
  $f^l_i:F_l(\Omega_l)\rightarrow B_{g_k}(p_k,l+1)$ are smooth maps for $i$ that maps $p_\infty$ to $p_i$ and are diffeomorphic onto their images so that $(f^l_k)^* g_i(t)\rightarrow g_\infty(t)$ locally smoothly uniformly on compact subsets of $F_l(\Omega_l)\times (0,S_l]$.   Moreover, $\Omega_l\supset B_{g_\infty(t)}(p_\infty, l)$ and
\be\label{e-GH-partial-2}
d_{g_\infty(s)}(x,y)-c_l(\sqrt t-\sqrt s)\le d_{g_\infty(t)}(x,y)\le e^{c_l(t-s)}d_{g_i(s)}(x,y).
\ee
for $ x,y\in  \in \bigcap_{\tau\in(0,S_l)} B_{g_\infty}(p_\infty, l/2)\subset \Omega_l$  for some possible larger $c_l$.

Together with \cite[Lemma 3.1, Corollary 3.3]{SimonTopping2016}, one can modify  the proof of   Theorem \ref{t-intro-GH} (\textbf{II}) to conclude that the distance function $d_\infty(s)$  induced by $g_\infty(s)$ will converge to a distance function $d_\infty$. Moreover,
   \be\label{e-GH-partial-3}
d_\infty (x,y)-c_l\sqrt t \le d_{g_\infty(t)}(x,y)\le e^{c_lt}d_\infty(x,y).
\ee
for $ x,y \in B_{d_\infty}(p_\infty,l/4)$, say.

Then $d_\infty$ and $d_\infty(s),s>0$ induce the same topology on $M_\infty$ thanks to Lemma \ref{Lemma3.1ST}. As in the proof of Theorem \ref{t-intro-GH} {\textbf(II)}, one can conclude that after passing to a subsequence, $(M_i,d_{g_i},p_i)$ converge to $(M_\infty,d_\infty,p_\infty)$ in the PGH sense.

{\sc Step 3:} Building on the above construction, it remains to prove that $M_\infty$ admits a complex structure $J_\infty$ and $g_\infty(t)$ is \K with respect to $J_\infty$ for $t>0$. The main ingredient here is that we constructed $g_i(t)$ so that it preserved K\"ahlerity locally. From {\sc Step 2} and \eqref{BB-in}, we can restrict to a smaller set so that $f_i^m: B_{d_\infty}(p_\infty,m-1)\rightarrow B_{g_i}(p_i,m+1)\subset M_k$ are smooth maps that map $p_\infty$ to $p_i$ such that $(f_i^m)^*g_i(t)\rightarrow g_\infty(t)$ smoothly uniformly on any $[\delta,T_m]$, $\delta>0$. Now we can follow closely the argument in \cite[Chapter 3]{ChowBookI} to construct the complex structure. Consider the sequence of $(1,1)$ tensor $J_{m,i}=((f_i^m)^{-1})_*J_i(f_i^m)_*$ on $B_{d_\infty}(p_\infty,m-1)$. Clearly on $B_{d_\infty}(p_\infty,m-1)$,
\begin{equation}\label{cpx-str-1}
\left\{
\begin{array}{ll}
&\nabla^{(f_i^m)_*g_i(t)}J_{m,i}=0;\\
&(J_{m,i})^2=-\mathrm{Id}_{TM_\infty};\\
&(f_i^m)_*g_i(t) \circ J_{m,i} =(f_i^m)_*g_i(t).
\end{array}
\right.
\end{equation}

For any $\delta>0$, since we have $(f_i^m)_*g_i(t)\rightarrow g_\infty(t)$ smoothly uniformly on $B_{d_\infty}(p_\infty,m-1)\times [\delta,T_m]$, \eqref{cpx-str-1}, {\bf (b)} and Shi's estimate \cite{Shi1989} implies that $J_{m,i}$ are locally uniformly bounded in any $C^l$ norm with respect to $g_\infty(T_m)$. Hence by passing to subsequence in $k$ and Ascoli-Arzel\`a Theorem, $J_{m,i}$ converges to a smooth $(1,1)$ tensor $J_m$ in $C^\infty_{loc}$ on compact sets of $B_{d_\infty}(p_\infty,m-1)$. Moreover, by passing \eqref{cpx-str-1} to $J_m$, we have a sequence of locally defined almost complex structure $J_m$ on $B_{d_\infty}(p_\infty,m-1)$ such that
\begin{equation}\label{cpx-str-2}
\left\{
\begin{array}{ll}
&\nabla^{g(t)} J_m=0;\\
&(J_m)^2=-\mathrm{Id}_{TM_\infty};\\
&g_\infty(t) \circ J_m =g_\infty(t).
\end{array}
\right.
\end{equation}

By passing {\bf (b1)} to limiting solution, $g_\infty(t)$ is a pyramid Ricci flow solution such that for all $k\in \mathbb{N}$,
$$|\Rm(g_\infty(t))|\leq \frac{a_{k+1}}{t}\quad \text{on}\;\;B_{d_\infty}(p_\infty,k)\times (0,T_k].$$

This together with Shi's higher order estimates \cite{Shi1989} and \eqref{cpx-str-2} implies that for fixed compact set $\Omega=B_{d_\infty}(p_\infty,k)$ and  $m$ sufficiently large, $J_m$ are uniformly bounded in any $C^l_{\Omega}$ norm with respect to metric $g(T_k)$. Therefore by diagonal subsequence argument and Ascoli-Arzel\`a Theorem again, $M_\infty$ admits a smooth almost complex structure $J_\infty=\lim_{m\rightarrow +\infty}J_m$.

To see that $J_\infty$ is a complex structure, it suffices to point out that the Nijenhuis tensor vanishes. This follows easily from the local existence of almost Hermitian metric $h=g_\infty(T_m)$ on each $B_{d_\infty}(p_\infty,m-1)$ so that $\nabla^h J_\infty=0$ where $\nabla^h$ is the Levi-Civita connection of $h$. The K\"ahlerity of $g_\infty(t)$ follows immediately.
\end{proof}

\begin{rem}
The general version of Gromov compactness theorem will imply the existence of limit space if the $\Ric(g_i)$ is bounded from below by a function of distance function $d_{g_i}(x,p_i)$, for example see \cite[Corollary 30]{Petersen2006}. In this case, it is also clear from the proof in \cite[Theorem 5.1]{McLeodTopping2018} that one may also allow the $\OB(g_i)$ and $\Ric(g_i)$ to be bounded from below by a uniform decreasing function of $d_{g_i}(x,p_i)$.
\end{rem}

\subsection{Applications to  \K manifolds with non-negative curvature}

First, we establish a longtime existence result of the \KR flow under maximal volume growth condition.
\begin{thm}\label{t-existence-NOB}
Suppose $(M^n,g_0)$ is a complete non-compact \K manifold with non-negative orthogonal bisectional curvature, non-negative Ricci curvature and maximal volume growth. There is a complete solution to the \KR flow $g(t)$ with $g(0)=g_0$ on $M\times[0,\infty)$ such that
 \begin{enumerate}
   \item [(i)] $\OB(g(t))\ge0, \Ric(g(t))\ge 0$ for all $t\ge 0$;
   \item [(ii)] $|\Rm(g(t))|\le at^{-1}, \mathrm{inj}_{g(t)}(x)\ge \sqrt{a^{-1}t}$ for some $a>0$ for all $x\in M, t>0$.
   \item [(iii)] $V_{g(t)}(B_{g(t)}(x,r))\ge vr^{2n}$ for some $v>0$ for all $x\in M, \; r>0$ and $t>0$.
 \end{enumerate}

\end{thm}

\begin{proof} Since $(M,g_0)$ has maximal volume growth, $V_{g_0}(B(x,r))\ge v r^{2n}$ for some $v>0$ for all $x\in M, r>0$. For any $R>>1$, let $ h^R_0=R^{-2}g_0$. Then $\OB(h^R_0)\ge 0, \Ric(h^R_0)\ge 0$ and $V_{h^R_0}(B(x,1))\ge v$ for all $x$ by volume comparison.  By Theorem \ref{KRF-exist}, there exist $a>0, T>0$ independent of $R$ so that  a short-time solution $  h^R(t)$ to the \KR flow defined on $M\times [0,T]$ $h^R(0)=h^R_0$ such that  and for all $(x,t)\in M\times (0,T]$,
\begin{equation}
\left\{
\begin{array}{ll}
&|\Rm(h^R(t))|\leq a t^{-1};\\
&\mathrm{inj}_{h^R(t)}(x)\geq \sqrt{a^{-1}t}
\end{array}
\right.
\end{equation}
By Corollary \ref{c-NOB}, we conclude that $\OB(h^R(t))\ge0, \Ric(h^R(t))\ge 0$.
 Rescale it back to $g_R(t)=R^2h^R(R^{-2}t)$. Then $g_R(t)$ is a solution to the \KR flow  so that it is defined on $M\times [0,T\cdot R^2]$, $g_R(0)=g_0$ and satisfies
\begin{equation}
\left\{
\begin{array}{ll}
&|\Rm( g_R(t))|\leq at^{-1};\\
&\text{inj}_{g_R(t)}(x)\geq \sqrt{a^{-1}t}\\
&\OB(g_R(t))\ge 0, \Ric(g_R(t))\ge0.
\end{array}
\right.
\end{equation}
for $(x,t)\in M\times(0, T\cdot R^2]$. As in the proof of Theorem \ref{KRF-exist}, we may find  $R_i\rightarrow +\infty$ so that   $g(t)=\lim_{i\rightarrow +\infty} g_{R_i}(t)$ is defined on $M\times [0,+\infty)$. And the convergence is locally uniformly in any $C^k$ norm on any compact sets of $M\times[0,T]$. Moreover, $g(t)$ satisfies (i) and (ii) in the theorem.

On the other hand, by \cite[Lemma 2.3]{SimonTopping2016}, by choosing a smaller $T$, one can conclude that $V_{h^R(t)}(x,1)\ge c>0$ for some $c$ independent of $R, t, x$. From this and the volume comparison, one can conclude that (iii) is true.
\end{proof}

{ Now we are ready to prove the uniformization result.
\begin{cor}\label{c-uniformization-1}
Suppose $(M^n,g_0)$ is a complete non-compact \K manifold with non-negative orthogonal bisectional curvature, non-negative Ricci curvature and maximal volume growth. Then $M$ is biholomorphic to a pseudoconvex domain in $\mathbb{C}^n$ which is homeomorphic to $\mathbb{R}^{2n}$. Moreover, $M$ admits a non-trivial holomorphic function with polynomial growth.
\end{cor}
\begin{proof}
Let $g(t)$ be as in Theorem \ref{t-existence-NOB}.
From the injectivity radius lower bound and non-negative Ricci curvature in the theorem, one can conclude that $M$ is simply connected. By taking $g_i=i^{-1}g(i)$ and by \cite[Theorem 1.2]{ChauTam2008}, $M$ is biholomorphic to a pseudoconvex domain in $\mathbb{C}^n$ which is homeomorphic to $\mathbb{R}^{2n}$. One can also prove the biholomorphism using the argument in \cite[section 9]{Shi1997}, see also \cite[Theorem 1.2]{ChenZhu2003}.

It remains to show the existence of non-trivial holomorphic function. Since $g(1)\leq g(0)$, we can assume $|\Rm(t)|\leq a(t+1)^{-1}$ for all $t>0$ by working on $g(1)$ instead of $g_0$. Let $F(x,t)= \log \frac{\det g(t)}{\det g_0}$ and $k(x,s)=\aint_{B_{g_0}(x,s)} R_{g_0} \,d\mu_{g_0}$. By the proof of \cite[Theorem 2.1]{NiTam2003}, curvature estimates $|\Rm|\leq a(t+1)^{-1}$ and Ricci flow equation, we can find $C>0$ such that for all $t>0$,$\int^t_0 sk(x,s) ds \leq C \log (t+1)$ and hence $\int_0^\infty k(x,s)ds <+\infty$. Using the result in \cite[Theorem 1.2]{NiTam2013}, we can find smooth function $u$ with logarithmic growth such that $\ddb u=\Ric(g_0)$. Let $v(t)$ be the heat equation with initial data $u$. Note that $M$ is simply connected. By \cite[Theorem 2.1, Proposition 2.1]{NiNiu2019}, there is splitting $M=M_1\times M_2$ where $\ddb v>0$ on $M_1$ and $\ddb v=0$ on $M_2$. By Cheng-Yau's gradient estimate, $v(x,t)=c=u(x)$ on $M_2$. Hence $M_2$ has nonnegative $\OB$ and $\Ric\equiv 0$. Therefore $M_2$ must be flat, see \cite[Page 9-10]{NiNiu2019}. This can also be done by applying Hamilton's strong maximum principle on $\Ric(g(t))$ directly as $\OB(g(t))\geq 0$. Therefore, there is a strictly plurisubharmonic function on $M$ with logarithmic growth. The existence now follows from standard argument using $L^2$ estimate of $\bar \partial$ equation \cite[Proposition 3.2]{Ni1998}, see \cite[Corollary 6.2]{NiTam2003-2} for example. This completes the proof.
\end{proof}

}

\appendix

\section{Some auxiliary results}
In this section, we collect some preliminary results for the construction of local \KR flow. First, we have a local existence lemma of \KR flow from \cite[Lemma 5.1]{LeeTam2017} which is based on using the Chern-Ricci flow. This is a \K analogy of \cite[Lemma 4.3]{SimonTopping2017}, see also \cite[Lemma 6.2]{Hochard2016}.

\begin{lma}\label{l-extension-1} There exists $1>\a_n>0$ depending only on $n$ so that the following is true: Let $(N^n,h_0)$ be a \K manifold and  $U\subset N$  is a precompact  open set. Let $\rho>0$ be such that $B_{h_0}(x,\rho)\subset\subset N$, $|\Rm(h_0)|(x)\le \rho^{-2}$ and $  \mathrm{inj}_{h_0}(x)\ge \rho$ for all $x\in U$. Assume $U_\rho$ is non-empty. Then for any component $X$ of $U_\rho$, there is a solution $h(t)$ to the \KR flow  on  $X\times[0,\a_n \rho^2]$, where for any $\lambda>0$
$$
U_\lambda:=\{x\in U|\ B_{h_0}(x,\lambda)\subset\subset U\},
$$
with $g(t)$ satisfies the following:
\begin{enumerate}
  \item [(i)] $h(0)=h_0$ on $X$
  \item [(ii)]
  $$
  \a_n h_0\le h(t)\le \a_n^{-1}h_0
  $$
  on $X\times[0,\a_n\rho^2]$.
\end{enumerate}
\end{lma}

  The following local estimates of \KR flow are by Sherman-Weinkove \cite{ShermanWeinkove2012} and Lott-Zhang \cite{LottZhang2013}. The following is from \cite[Propositon A.1]{LottZhang2013}.
\begin{lma}\label{l-curv1}
For any $K,n>0$, there exists $A(n,K)$ depending only on $n, K$ such that the following holds: For any \K manifold $(N^n,g_0)$ (not necessarily complete), suppose $g(t), t\in [0,S]$ is a solution of \K Ricci flow on $B_{g_0}(x_0,r)$ with $g(0)=g_0$ and $B_{g_0}(x_0,r)\Subset N$ such that on $B_0(x_0,r)$,
$$|\mathrm{Rm}(g_0)|\leq Kr^{-2} \quad\text{and}\quad |\nabla_{g_0} \Rm(g_0)|\leq Kr^{-3}. $$
Assume in addition that on $B_{g_0}(x_0,r)$,
$$K^{-1}g_0\leq g(t)\leq K g_0$$
Then on $B_{g_0}(x_0,\frac r8)$, $t\in [0,S]$,
$$|\Rm|(g(t))\leq Ar^{-2}.$$

\end{lma}

{ In the study of Ricci flow which may not be complete or may not have uniformly  bounded curvature, we need good estimates to compare  distance functions in different time. Here are basic results which are used in this work. The first one is the shrinking balls lemma by Simon-Topping   \cite[Corollary 3.3]{SimonTopping2016}}:

\begin{lma}\label{l-balls}
 There exists a constant $\beta=\beta(m)\geq 1$ depending only on $m$ such that the following is true. Suppose $(N^m,g(t))$ is a Ricci flow for $t\in [0,S]$ and $x_0\in N$   with $B_0(x_0,r)\subset\subset M$ for some $r>0$, and $\Ric(g(t))\leq (n-1)a/t$ on $B_{g_0}(x_0,r)$ for each $t\in (0,S]$. Then
$$B_t\left(x_0,r-\beta\sqrt{a t}\right)\subset B_{g_0}(x_0,r).$$
\end{lma}

The following bi-H\"older Distance estimates are from \cite[Lemma 3.1]{SimonTopping2017}.
\begin{lma}\label{Lemma3.1ST}
Let $(N^m,g(t))$ be a Ricci flow for $t\in(0,T]$, not necessarily complete, and let $p\in N$  such that for all $t\in (0,T]$, we have $B_{g(t)}(p,2r)\Subset  N$. Suppose further that for some $c_0,\a>0$ and for each $t\in (0,T]$, we have $$-\a\leq \Ric(g(t))\leq \frac{(m-1)c_0}t $$ on $B_{g(t)}(x_0,2r)$. Define $\Omega_T=\cap_{t\in(0,T]}B_{g(t)}(x_0,r)$. For any $x,y\in \Omega_T$ let $d_{g(t)}(x, y)$ be the distance of $x, y$ with respect to $g(t)$ which  is unambiguous defined for all $t\in(0,T]$ and must be realised by a minimising geodesic lying within $B_{g(t)}(x_0,2r)$. Then we have the following.
\begin{enumerate}
  \item [(1)] For any $0<s\leq t\leq T$, we have
$$d_{g(s)}(x,y)-\b(m) \sqrt{c_0}(\sqrt{t}-\sqrt{s})\leq d_{g(t)}(x,y)\leq e^{\a(t-s)}d_{g(s)}(x,y) $$
for some positive constant $\b(m)$ depending only on $m$.
  \item [(2)] $d_{g(s)}$ converges uniformly to a distance metric $d_0$ on $\Omega_T$ as $t\rightarrow 0^+$ and
$$d_{0}(x,y)-\b(m) \sqrt{c_0}\sqrt{t}\leq d_{g(t)}(x,y)\leq e^{\a t}d_{0}(x,y).$$
  \item [(3)]  There exists $\gamma>0$, depending only on $n$, $c_0$ and upper bounds for $T$ and $r$, such
that $d_{g(t)}(x,y)\geq \gamma (d_0(x,y))^{1+2(m-1)c_0}$ for all $t\in (0,T]$.
\item[(4)] For all $t\in (0,T]$ and $R<R_0 =:re^{-\a T} -\b\sqrt{c_0T}$, we have
$$B_{g(t)}(x_0,R_0)\subset \Omega_T;\ B_{d_0}(x_0,R)\Subset  \mathcal{O}$$  where $\mathcal{O}$ is the component of   $int(\Omega_T)$ containing $x_0$.
\end{enumerate}

\end{lma}


\section{An example by Ni-Zheng}

In \cite[section 7]{NiZheng2018}, Ni and Zheng has constructed an $U(n)$ invariant \K metric on $\C^n$ with $\OB\ge0,\Ric>0 $ and holomorphic sectional curvature being negative somewhere. In particular, the holomorphic bisectional curvature is not nonnegative.   In this section, we will show that their example also have maximal volume growth. Let $(z_1,...,z_n)$ be the standard coordinate on $\mathbb{C}^n$ and $r=|z|^2$. An $U(n)$ invariant metric on $\mathbb{C}^n$ can be characterized by a smooth function $\xi\in C^\infty([0,+\infty)]$with $\xi(0)=0$ in the following way: for a given smooth function $\xi$, define $h(r)=\exp\left( -\int^r_0s^{-1}\xi(s)ds\right)$ and $f(r)=r^{-1}\int^r_0h(s)ds$ with $h(0)=1$. Then $g_{ij}=f(r)\delta_{ij}+f'(r) \bar z_i z_j$ defines a $U(n)$-invariant \K metric on $\mathbb{C}^n$. In particular, $g$ is a complete metric if $\xi\in (0,1)$ according to \cite{WuZheng}.

First, we recall a necessary and sufficient condition for a complete $U(n)$ invariant \K metric to have asymptotic Euclidean volume growth using the characterization function $\xi$. The following theorem was proved in \cite[Theorem 3]{WuZheng} within the class of $U(n)$ invariant metrics with nonnegative bisectional curvature, but it is clear from their proof that it suffices to assume the completeness, namely $\xi \in (0,1)$.
\begin{thm}\cite[Theorem 3]{WuZheng}
If $\xi \in (0,1)$, then the corresponding complete $U(n)$ invariant \K metric $g$ satisfies $$\lim_{r\rightarrow +\infty}\frac{V_{g}(B_g(p,r))}{r^{2n}}=c_n(1-\xi(\infty))^{2n}$$
for some dimensional constant $c_n>0$.
\end{thm}

It suffices to check that the example given in \cite[Section 7]{NiZheng2018} satisfies $\xi(\infty)<1$. Following the notations in \cite{NiZheng2018}, we have
$$\frac{1}{(1-\xi(r))^2}=1+(F'(x))^2=(k(x^2))^2$$
where $F$ is a function of $x=\sqrt{rh}$. Hence, the condition $\xi(\infty)<1$ is equivalent to the finiteness of $F'(x_0)$ where $x_0=\lim_{r\rightarrow +\infty}\sqrt{rh}\leq +\infty$. From their first example, $k(t)$ is chosen to be $1+\a+t\a'$ where
$$\a(t)=\lambda \left(1-\frac{1}{(1+t^2)^a}\right)$$
for some $a\in (\frac{1}{2},1)$ and $\lambda=\lambda(a)>1$. In particular, for $\lambda$ sufficiently large, the corresponding $U(n)$-invariant \K metric will have $\OB,\Ric> 0$ but negative holomorphic sectional curvature somewhere. To check the finiteness of $F'(x_0)$,
\begin{align*}
\a+t\a' &= \lambda \left(1-\frac{1}{(1+t^2)^a}\right)+\frac{2a\lambda t^2}{(1+t^2)^{a+1}}\rightarrow \lambda <\infty.
\end{align*}
Hence $x_0=\infty$ and $\xi(\infty)<1$. As $g$ has $\Ric\geq 0$, $g$ has maximal volume growth.

\end{document}